\documentclass{article}
\usepackage[affil-it]{authblk}
\usepackage{geometry}
\usepackage{latexsym}
\usepackage{amsthm}
\usepackage{amssymb}
\usepackage{amscd}
\usepackage{amsmath}
\usepackage{pdflscape}
\usepackage{bm}
\usepackage[font=small,labelfont=bf]{caption}
\usepackage{graphicx}
\usepackage{epsfig}
\usepackage{epstopdf}
\usepackage{placeins}
\usepackage{booktabs}
\usepackage[caption=false]{subfig}
\usepackage{thmtools,thm-restate}
\makeatletter
\newcommand*\rel@kern[1]{\kern#1\dimexpr\macc@kerna}
\newcommand*\widebar[1]{%
  \begingroup
  \def\mathaccent##1##2{%
    \rel@kern{0.8}%
    \overline{\rel@kern{-0.8}\macc@nucleus\rel@kern{0.2}}%
    \rel@kern{-0.2}%
  }%
  \macc@depth\@ne
  \let\math@bgroup\@empty \let\math@egroup\macc@set@skewchar
  \mathsurround\z@ \frozen@everymath{\mathgroup\macc@group\relax}%
  \macc@set@skewchar\relax
  \let\mathaccentV\macc@nested@a
  \macc@nested@a\relax111{#1}%
  \endgroup
}
\makeatother

\providecommand{\R}{\mathbb{R}}

\providecommand{\Z}{\mathbb{Z}}
\providecommand{\N}{\mathbb{N}}
\providecommand{\vol}[1]{\text{vol}\left({#1}\right)}

\renewcommand{\vec}[1]{\boldsymbol{#1}}

\providecommand{\abs}[1]{\left\vert#1\right\vert}
\providecommand{\norm}[1]{\left\Vert#1\right\Vert}
\providecommand{\set}[1]{\left\{#1\right\}}

\providecommand{\paren}[1]{( #1 )}

\providecommand{\brac}[1]{\left [ #1 \right]}
\newcommand{\s}[1]{\begin{equation*} \begin{split} #1 \end{split} \end{equation*}}
\providecommand{\eqref}[1]{(\ref{#1})}

\newtheorem{thm}{Theorem}[section]
\newtheorem{theorem}[thm]{Theorem}
\newtheorem{lemma}[thm]{Lemma}
\newtheorem{conj}[thm]{Conjecture}

\newtheorem{definition}[thm]{Definition}
\newtheorem{Definition}[thm]{Definition}
\newtheorem{proposition}[thm]{Proposition}
\newtheorem{example}[thm]{Example}
\newtheorem{question}[thm]{Question}
\numberwithin{equation}{section}
\graphicspath{{.//}}

\begin{document}

\title{\Large Limits of Embedded Graphs \\ \vspace{10pt} \small and Universality Conjectures for the Network Flow}
\author{Benjamin Schweinhart}
\date{June 2017}

\maketitle

\begin{abstract}
We define notions of local topological convergence and local geometric convergence for embedded graphs in $\R^n,$ and study their properties. The former is related to Benjamini-Schramm convergence, and the latter to weak convergence of probability measures with respect to a certain topology on the space of embedded graphs. These are used to state universality conjectures for the long-term behavior of the network flow, or curvature flow on embedded graphs. To provide evidence these conjectures, we develop and apply computational methods to test for local topological and local geometric convergence.

\end{abstract}

\section{Introduction}
It has long been observed, in experiments and simulations, that scale-invariant statistics of grain boundary networks appear to be universal.~\cite{1986mullins,2012lazar} That is, the statistics of their long-term scale-invariant properties are largely independent of the initial conditions. This is particularly interesting because the universal condition is both statistical and transient - a finite graph that approaches it will not stop flowing, but will rather continue evolving until it has very few edges. This observation motivates the introduction of two notions of graph limit for embedded graphs, one topological and one geometric. We use these to formally state universality conjectures for the network flow on graphs. We computationally test the local topological and local geometric convergence of simulations of the network flow in two dimensions. 

In Section~\ref{sec_curvature_flow} we briefly introduce the network flow for embedded graphs in $\R^n,$ which is also known as curvature flow on embedded graphs. The planar case is the simplest model of the physical phenomenon of grain growth in a polycrystal~\cite{1956mullins}, and the evolution of graphs in higher dimensions is also of mathematical interest. Most of the concepts defined here can be extended to general $k$-dimensional regular cell complexes in $\R^n,$ and a similar universality phenomenon has been observed for the most physically important case of $2$-dimensional cell complexes in $\R^3.$  The eventual goal of the program outlined here would be to prove the existence of universal statistics in that case, and to study their properties. However, we will focus on embedded graphs in this paper for the purpose of clarity.

We define two notions of of convergence for embedded graphs - local topological, or Benjamini-Schramm, convergence in Section~\ref{sec_swatch} and local geometric convergence in Section~\ref{sec_geoConv}. The former implies convergence of all local topological statistics, and is connected to the method of swatches introduced previously in our previous paper~\cite{2016schweinhart,2012masonB}. The latter notion was developed in analogy with the former, and implies convergence of averages of local geometric properties, in a sense to be defined in this paper. It is weak convergence of probability measures on the space of embedded graphs with a smooth topology of disjoint topological types (Section~\ref{sec_topType}). This topology is not particularly nice or natural, but we give conditions in Theorem~\ref{thm_convergence} under which weak convergence of probability measures on the space of graphs with either the Hausdorff metric topology or the varifold topology implies local geometric convergence.

In Section~\ref{sec_computations}, we introduce computational methods to test for local topological and local geometric convergence, and apply them to simulations of the network flow on planar graphs. In contrast, previous papers appearing in the materials science and physics literatures have used a few ad-hoc measures to claim convergence of simulations to the conjectural universal state, before proceeding to study its properties. The notions presented here were developed in part to motivate a more systematic method to verify the convergence of such simulations.

We formalize universality conjectures for the network flow in Section~\ref{sec_conj}. The main hypothesis we propose is that the graphs are homogeneous. Homogeneous graphs are defined in~\ref{sec_hom}, and are roughly those whose local properties have well-defined spatial averages. Computational simulations indicate that these graphs either evolve to approach a universal state, or in very special cases become stationary with respect to the flow. The main conjectures are:

\begin{conj}[Universality Conjecture for Local Topological Convergence]
\label{steady_state_hypothesisb}
There exists a probability distribution $\sigma_n$ on the space of countable, connected graphs with a root vertex specified such that any network flow $G\paren{t}$ with homogeneous initial condition $G\paren{0}\in\mathcal{G}^n$ converges in the local topological sense to $\sigma_{\Omega}$ or to a stationary state as $t\rightarrow\infty.$
\end{conj}

\begin{conj}[Universality Conjecture for Local Geometric Convergence]
\label{steady_state_hypothesisa}
There exists a probability distribution $\Xi_n$ on $\mathcal{G}^n$ such that  that any network flow $G\paren{t}$ with homogeneous initial condition $G\paren{0}\in\mathcal{G}^n$converges in the scale-free local geometric sense  to $\Xi_{n}$ or a stationary state as $t\rightarrow\infty.$
\end{conj}

\begin{figure}
\center
\subfloat[]{%
	\label{ss_2D}{%
		\includegraphics[width=7cm]{%
			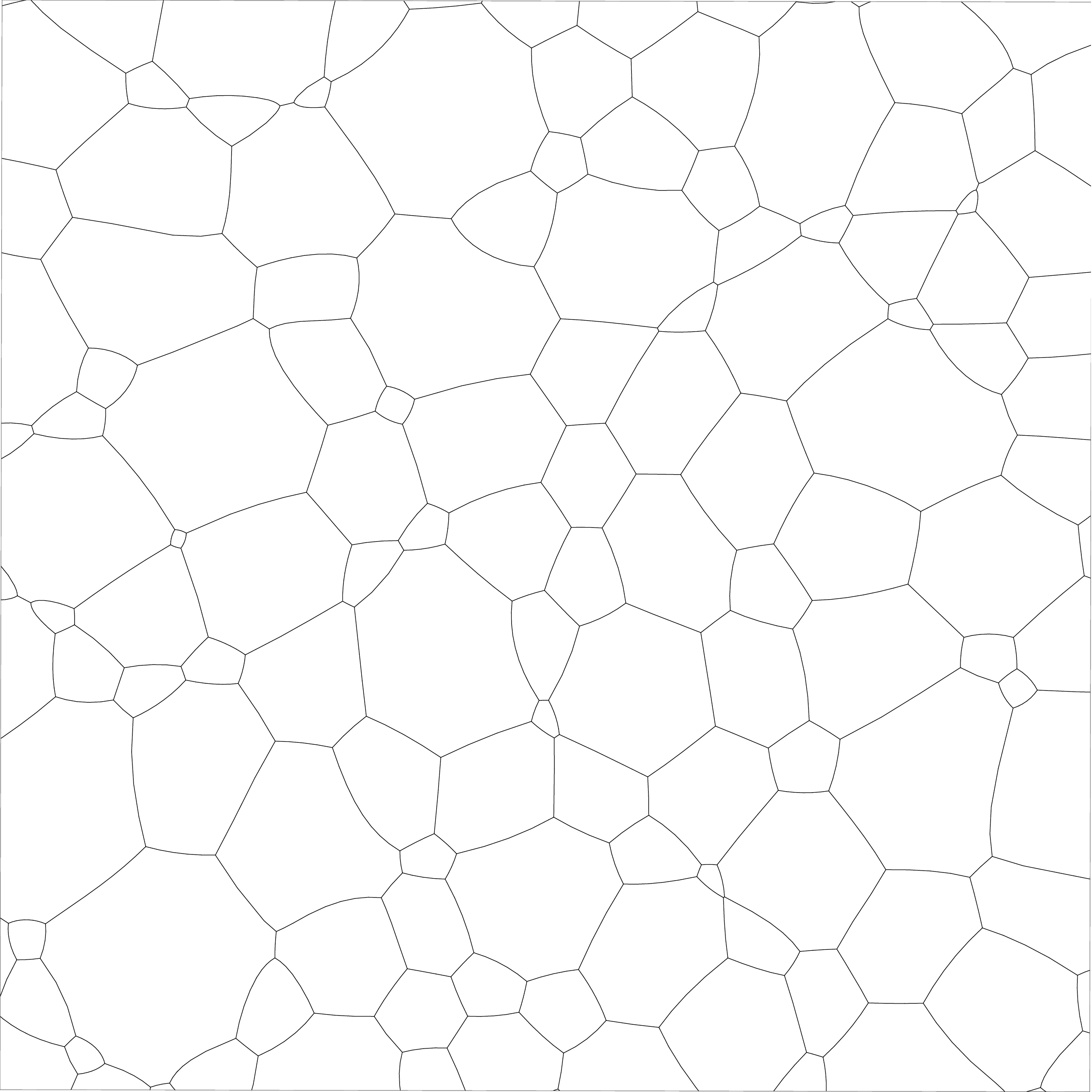}}}
\hspace{20pt}
\subfloat[]{%
	\label{ss_3D}{%
		\includegraphics[width=7cm]{%
			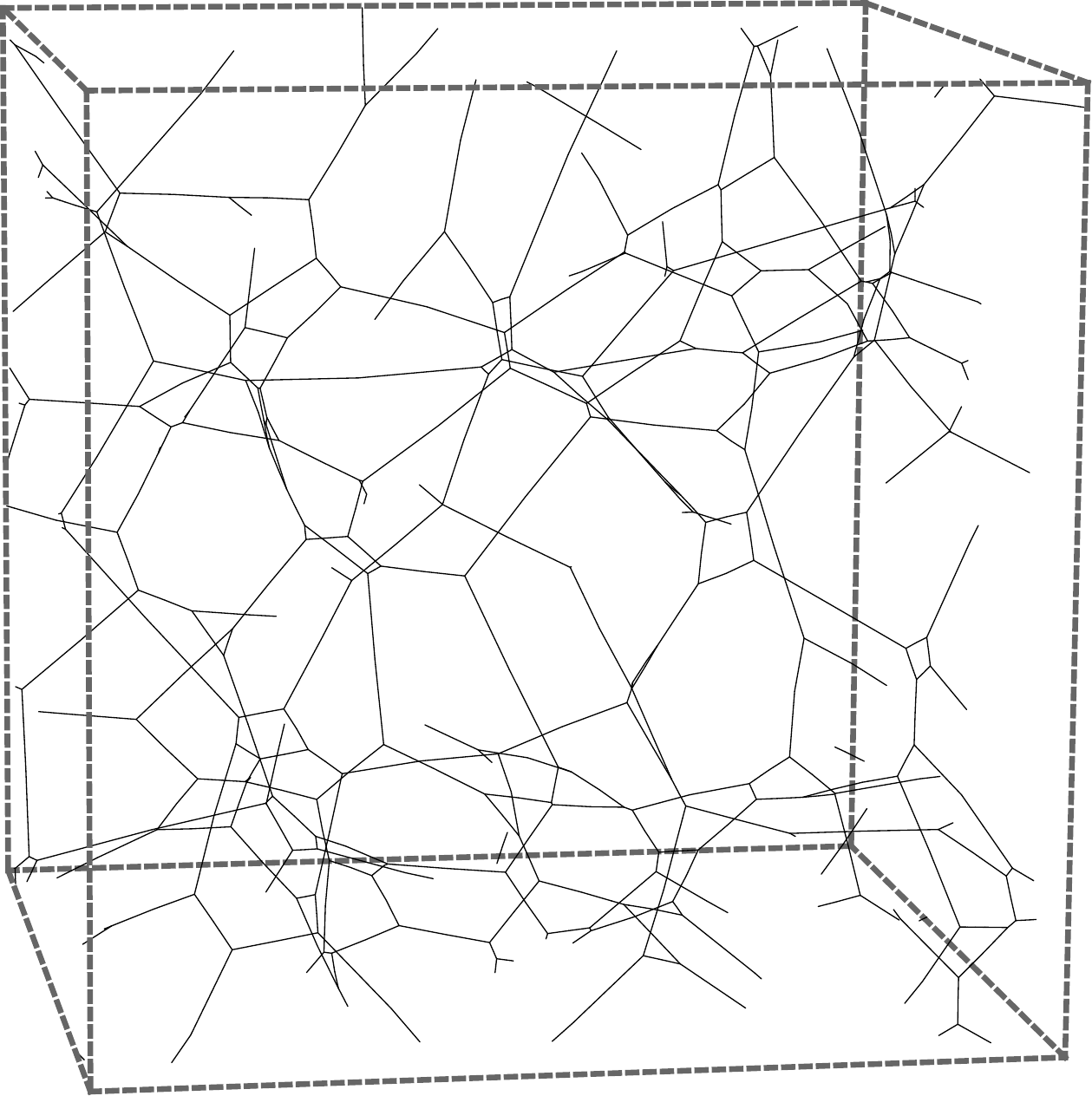}}}
\caption{Small regions of the conjectured universal state in (a) 2 and (b) 3 dimensions. They were produced by simulations discussed in Section~\ref{sec_computation}}
\end{figure}

A stationary state is one where the graph is at equilibrium with respect to the flow: its edges are straight and meet at trivalent vertices at angles of $2\pi/3.$ 

\subsection{Definition of Embedded Graphs}

Throughout, an embedded graph will be a locally finite collection of smooth curves in $\R^n,$ subject to certain niceness conditions:
\begin{Definition}
An \textbf{embedded graph} $G\subset K$ in a compact subset $K$ of $\R^n$ is a finite collection of smoothly embedded, connected curves with boundary with the following properties:
\begin{itemize}
\item The curves only meet at their boundaries, and are called the edges of the embedded graph. 
\item The boundary points of the curves are called vertices, and at least three curves meet at each vertex not in $\partial K.$
\item The outward-pointing unit tangent vectors of the curves meeting at a vertex are distinct.
\end{itemize} 
An embedded graph $G$ in an open subset $U$ of $\R^n$ is a collection of smoothly embedded, compact, connected curves with boundary so that $\widebar{G\cap K}$ is an embedded graph for any compact subset $K$ of $\R^n.$
\end{Definition} The second hypothesis implies that two embedded graphs in $\R^n$ are homeomorphic if and only if they are combinatorially equivalent. The intersection of an embedded graph in $\R^n$ with any open or compact subset of $\R^n$ is also an embedded graph, though new edges may be created in the process (as well as vertices at the boundary, if the set is compact) .In general, it is not always the case that $\widebar{U\cap G}=\bar{U}\cap G.$ The sets of embedded graphs in $\R^n$ and the open $n-$ball of radius $r$ are denoted $\mathcal{G}^n$ and $\mathcal{G}_r^n,$ respectively. We will define a topology for these sets in Section~\ref{sec_graphTop}.

\subsection{Summary of Concepts}
Here, we include a table summarizing the notions of graph limits defined in this paper, and a list of definitions of the terms appearing therein.
\begin{table}[ht]
  \footnotesize
  \centering
  \begin{tabular}{|c|c|c|c|}
    \hline\textbf{Convergence} & \textbf{Local Topological} & \textbf{Local Geometric} \\ \hline
    \textbf{Space of}  & Abstract Graphs &  Embedded Graphs in $\R^n$ \\ 
		& with Bounded Degree & with Smooth Geometric Cloth \\ \hline
    \textbf{Local} & Probability Distribution & Probability Distribution\\
		\textbf{Distributions} & on $\mathcal{S}_{r,k}$   &  on $\mathcal{G}_{r}^n$   \\ \hline
    \textbf{Universal} & Probability Distribution & Probability Distribution\\ \hline
		\textbf{Distribution}& on $\mathfrak{G}^{\bullet}$ & on $\mathcal{G}^n$\\ \hline
    \textbf{Symmetry} & Involution Invariance  & Translation Invariance \\ \hline
		\textbf{Addnl. Symmetry} & Topological Homogeneity & Homogeneity (Ergodicity of Translations) \\
		\textbf{Convergence} & Convergence of discrete & Weak Convergence of\\ 
		& Prob. Distributions & Unif. Separating Sequence \\ \hline
		\textbf{Controls} & Local Topological Stats & Local Geometric Stats \\ \hline
		
  \end{tabular}
	\end{table}  
	\normalsize
	\begin{enumerate}
		\item Local topological convergence - this is the same notion as Benjamin-Schramm convergence. We use both terms interchangeably, to emphasize the analogy with local geometric convergence. Defined in Section~\ref{sec_convergence}.
	\item $\mathcal{S}_{r,k}$ - the set of combinatorial isomorphism classes of rooted graphs. Defined in Section~\ref{sec_cloth}.
	\item  $\mathfrak{G}^{\bullet}$ - the space of abstract graphs with a root vertex specified. Defined in Section~\ref{sec_convergence}.
  \item Involution invariance - roughly that a probability distribution on $\mathfrak{G}^\bullet$ gives equal weight to different root vertices of the same graph. 
		\item Local topological property - Definition~\ref{local_topo_prop}. Roughly, any property that can be expressed in terms of subgraph frequencies.
	\item  Topological homogeneity - Definition~\ref{defn_tophomo}. That local topological statistics of an embedded graph in $\R^n$ are well-defined over more general sets than balls.
	\item $\mathcal{G}_{r}^n$ ($\mathcal{G}^n$) - the space of embedded graphs in the open $n$-ball of radius $r$ ($\R^n$), with the local Hausdorff metric topology. 
	\item Smooth geometric cloth - Definition~\ref{def_geo_cloth}. A probability distribution on $\mathcal{G}^n$ reflecting the averages of local geometric properties of a graph over balls of increasing radius.
	\item Uniformly separating sequence - Definition~\ref{defn_unifSep}. A sequence of probability measures on $\mathcal{G}^n$ whose local distributions are tight with respect to a certain collection of compact subsets of $\mathcal{G}^n.$
	\item Local geometric property - Definition~\ref{def_local_prop}. A bounded, continuous function on the space of embedded graphs in the ball of radius $r$ with the smooth topology of topological types.
		\item Homogeneity - Definition~\ref{defn_homo}. A graph $G\in\mathcal{G}^n$ is homogeneous if the $\R^n$ translation action is ergodic with respect to its geometric cloth.
	\end{enumerate}

\section{Motivation: Curvature Flow on Graphs}
\label{sec_curvature_flow}

We provide a brief introduction to the network flow, which is curvature flow on embedded graphs in $\R^n.$ A full discussion of the current state of mathematical knowledge about this flow is beyond the scope of this paper, and we direct readers to the comprehensive survey by Mantegazza, Novaga, Pluda, and Schulze~\cite{MNPS16}. Instead, we focus on a qualitative description of the properties observed in simulations, in order to motivate the graph limits defined later in the paper. We should note that the mathematical analysis of these systems is very difficult, but significant progress has recently been made for the case of graphs embedded in two dimensions~\cite{2014ilmanen}.

\begin{figure}
\center
\includegraphics[width=.9\textwidth]{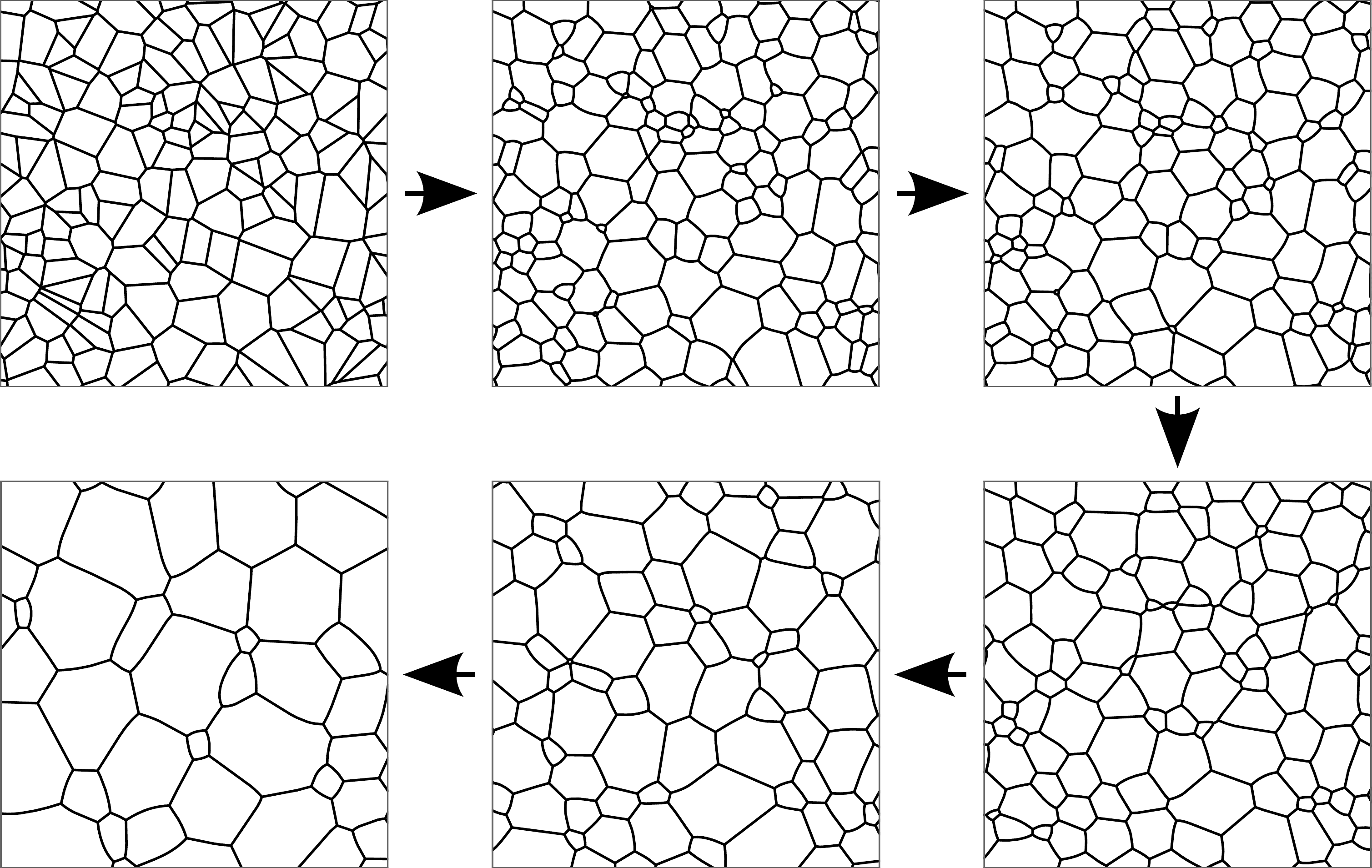}
 \caption{\label{fig:ggEvolution} A planar graph evolving by curvature flow.}
\end{figure} 

We  define curvature flow on embedded graphs by locally prescribing the behavior at each point, following the approach in ~\cite{2014ilmanen}. Alternatively, it can be defined as a type of ``gradient flow'' on the sum of the lengths of the edges~\cite{2014elsey}, but we omit that description here for simplicity.

A \textbf{regular graph} is an embedded graph such that every vertex has degree three, and the angles between the outward pointing tangent vectors of any edges meeting at a vertex equal $\frac{2\pi}{3}$ at that vertex. The latter condition is called the Herring Angle Condition. For such graphs, curvature flow is defined as follows:

\begin{Definition}
A continuous function $f:\paren{T_1,T_2}\rightarrow \mathcal{G}^n$ is a \textbf{curvature flow on regular graphs} if $f\paren{t}$ is a regular graph for all $t\in\paren{T_1,T_2},$ and there there exists a collection of parametrizations of the edges of $f\paren{t},$ $\set{\gamma_i:\paren{T_1,T_2}\rightarrow\R^n}$, satisfying
\s{\frac{\partial \gamma_i\paren{t}}{\partial t}=\kappa\vec{n}\;\;\;\; \forall t\in\paren{T_1,T_2}}
where $\kappa$ is the curvature and $\vec{n}$ is the unit normal vector with orientation given by the parametrization.
\end{Definition}
The restriction to regular graphs is quite severe, but the definition of the the network flow can be extended to more general graphs as follows:
\begin{Definition}
A continuous function $f:\brac{T_1,T_2}\rightarrow \mathcal{G}^n$ is a \textbf{curvature flow on graphs} and a \textbf{network flow} if for any bounded region $B\subset\R^n,$ there is a finite collection of times $\set{T_1=s_1,\ldots,s_m=T_2}$ such that $f_B\paren{t}:=f\paren{t}\cap B$ defines a curvature flow on regular graphs on the interval $\paren{s_i,s_{i+1}}$ and 
\s{\lim_{t\rightarrow s_i} f\paren{t}=f\paren{s_i}}
for each $1\leq i \leq n,$ in the varifold topology. The $\set{s_i}$ are called singular times.
\end{Definition}
\cite{2014ilmanen} discusses hypotheses under which a network flow exists starting from a non-regular initial condition in two dimensions.Note that the behavior of the flow after a singular time is not necessarily uniquely determined.

\subsection{Topological Changes}
A key feature of curvature flow on graphs is the existence of singularities in time which result in topological changes. The simplest of these changes, occurs when an edge shrinks to a point and two trivalent vertices collide, resulting in a vertex of degree four. This vertex will split into two vertices of degree three, with the adjacent edges shuffled.  Another topological change occurs when a triangle shrinks to a point, resulting in the deletion of three edges and three vertices. Ilmanen, Neves and Schulze~\cite{2014ilmanen} conjecture that all singularities of the network flow in two dimensions occurs when when a connected collection of edges shrinks to a point to form a vertex of degree greater than or equal to three. In particular, the flow never leads to an edge with multiplicity greater than one.

\subsection{Qualitative Behavior}

\begin{figure}
\center
\subfloat[]{%
	\label{fig:vor}{%
		\includegraphics[width=7cm]{%
			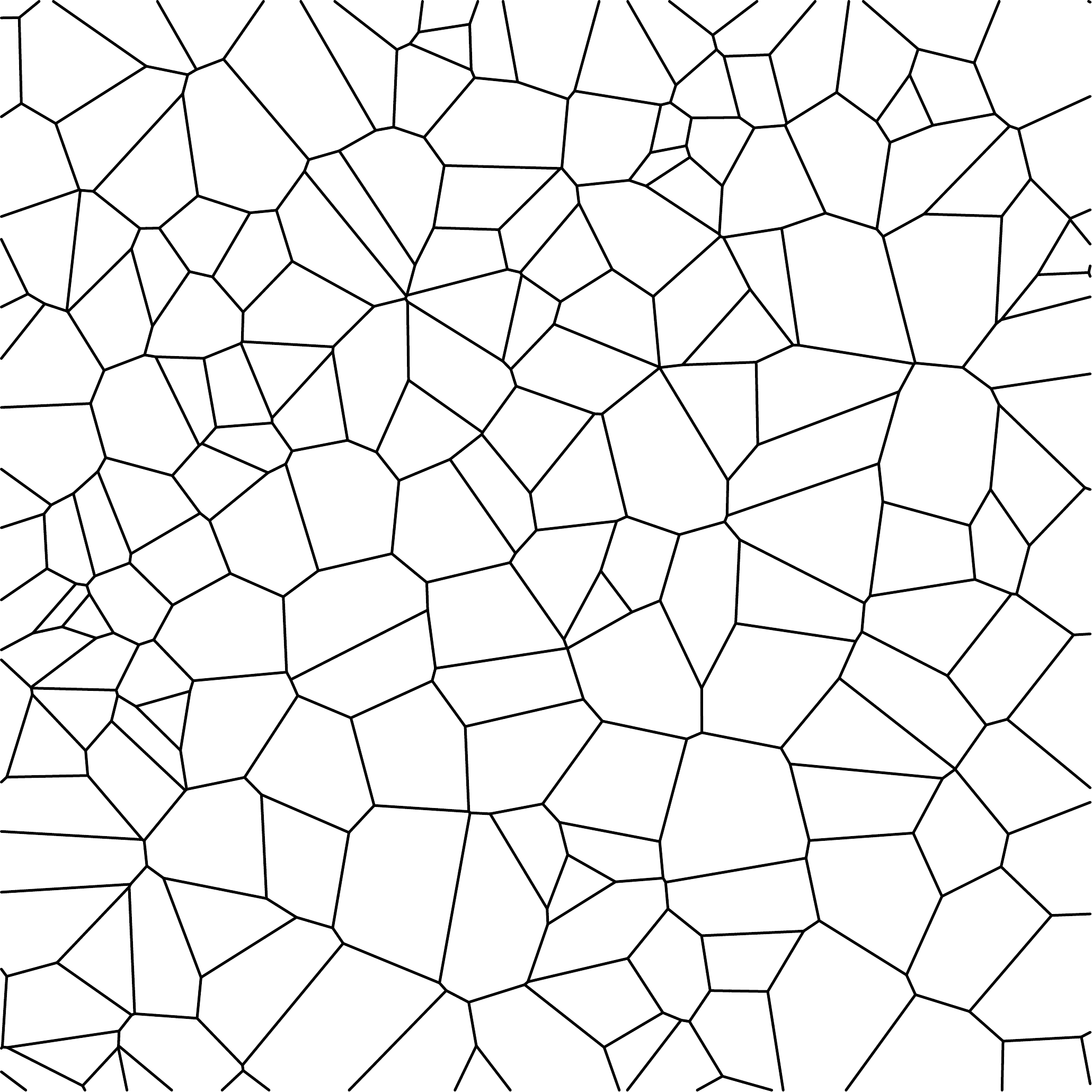}}}
\hspace{20pt}
\subfloat[]{%
	\label{perturbed_lattice}{%
		\includegraphics[width=7cm]{%
			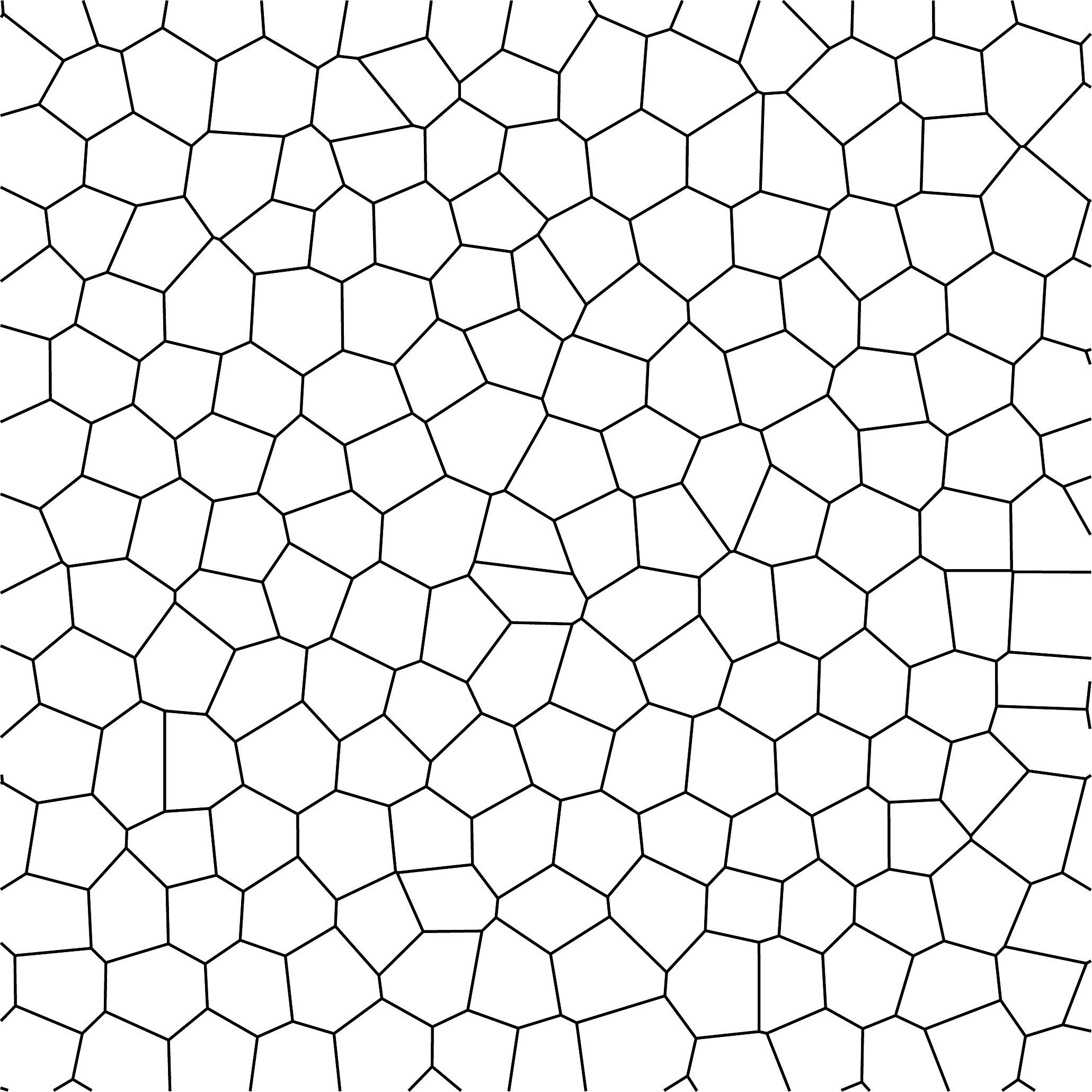}}}
\caption{\label{fig_initialConditions}Two initial conditions that appear to exhibit universal behavior under curvature flow: (a) a Voronoi diagram of Poisson distributed points and (b) a perturbed hexagonal lattice.}
\end{figure}

A mathematical understanding of the long-term behavior of the network flow is beyond current knowledge. However, materials scientists and physicists have put considerable effort toward simulating these systems. Here, we briefly discuss the observed behavior for graphs embedded in two dimensions.

For many random initial conditions, such as those depicted in Figure~\ref{fig_initialConditions}, the network flow appears to exhibit universal behavior in the long-term. That is, the statistics of their long-term scale-invariant properties converge to values that are independent of the initial conditions. We will show evidence for this in Section~\ref{sec_computations}. Another important behavior of these systems is that the average edge length is observed to decrease at a rate of $t^{-1/2}.$. This causes the area of components of the complement (or ``grains'') to increase, in a process called coarsening. Simulated systems of periodic graphs coarsen until until they reach a very small stationary state.

Stationary states of the network flow are embedded graphs with straight line edges meeting at trivalent vertices at angles of $2\pi/3.$ An example is the regular hexagonal lattice. There also exist non stationary flows that share the local statistics a stationary state. For example, an infinite hexagonal lattice with an edge removed will have one grain that expands forever, but the graph will always look hexagonal sufficiently far away from it.

Finally, one can construct examples that neither become stationary nor appear to exhibit universal behavior. Two such initial conditions are shown in Figure~\ref{fig_halfHex2}. We introduce the hypothesis of homogeneity in Section~\ref{sec_hom} to rule out this behavior.

\begin{figure}
\subfloat[]{%
	\label{fig_halfHex2}{%
		\includegraphics[width=7cm]{%
			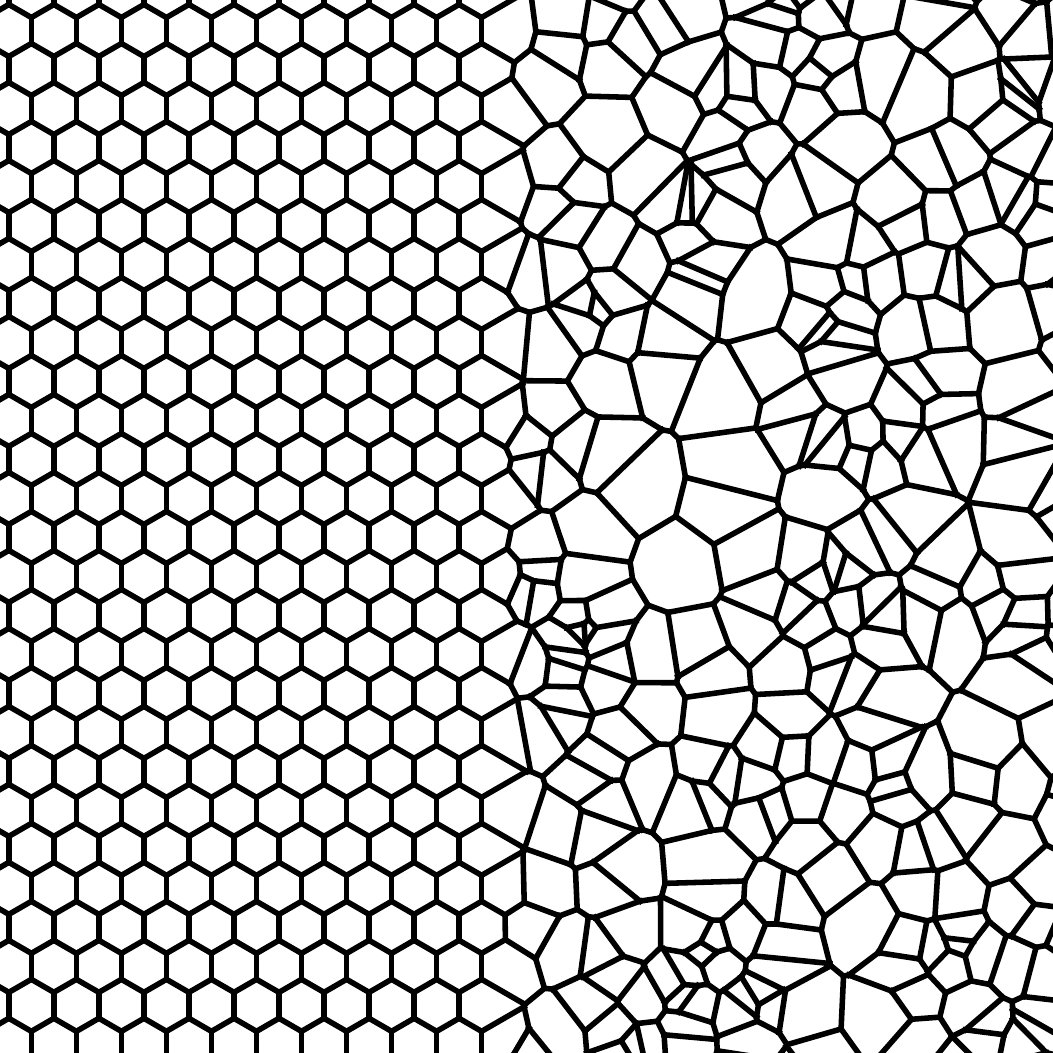}}}
\hspace{20pt}
\subfloat[]{%
	\label{fig_1Din2D}{%
		\includegraphics[width=7cm]{%
			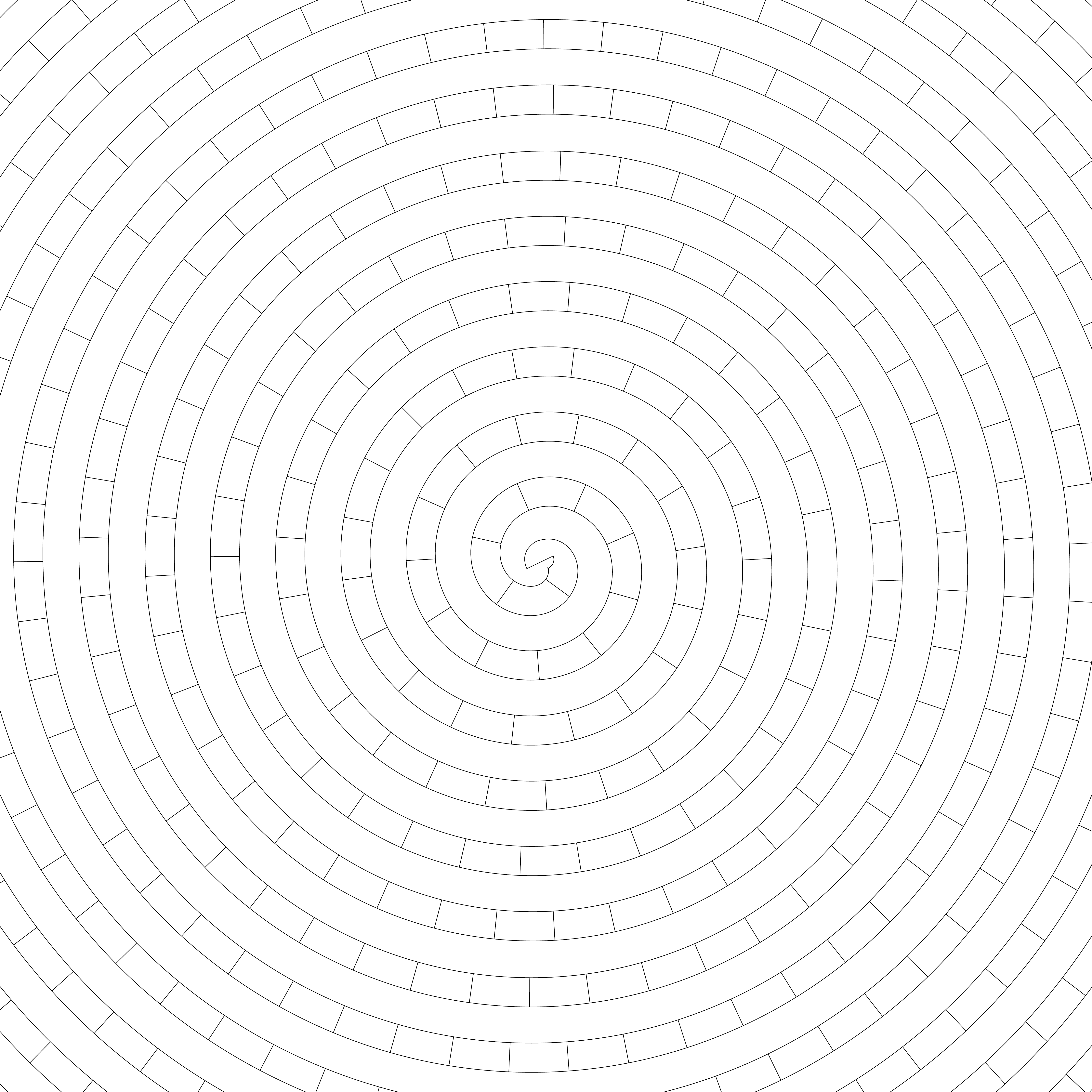}}}
\caption{\label{fig_inhomogenous} Finite regions of two graphs that do not appear to exhibit universal behavior. }
\end{figure}

\section{Local Topological Convergence}
\label{chp_Swatches}
\label{sec_swatch}
Local topological convergence, or Benjamini-Schramm convergence, is a notion of weak local limit for abstract graphs. It associates to a graph a family of probability distribution of local topological configurations of radius $r$ for each $r\in\N.$ Local topological convergence is the convergence of these discrete probability distributions. It was first noted that curvature flow on graphs could correspond to a Benjamini-Schramm graph limit in our paper ``Topology of Random Cell Complexes, and Applications''~\cite{2016schweinhart}, building on the work in~\cite{2012masonB}. In that paper, we defined local topological convergence for general regular cell complexes, but here we briefly review the concept for graphs.

\subsection{Swatches and Topological Cloths}
\label{sec_cloth}

In the following, the graph distance between two cells (edges or vertices) of a graph $G$ is the minimum number of edges and vertices in a path connecting them. The graph distance from a vertex to a neighboring edge is one, between vertices sharing an edge is two, and between cells in different connected components is $\infty.$ . A swatch is a ball in the graph distance:

\begin{Definition}
Let $v$ be a vertex of a graph $G.$ The \textbf{swatch} of radius $r$ at $v$ is the ball of radius $r$ centered at $v$ in the graph distance.
\end{Definition}

Note that the swatch is not a graph on vertices contained in $G,$ but rather a bipartite graph on the set of edges and vertices of $G.$ That is, an edge may be contained in a swatch that does not contain both of its adjacent vertices. Allowing such intermediate neighborhoods does not matter for the theory of topological convergence, but it gives more flexibility for computation applications.

We will be counting graph isomorphism classes of swatches. Two swatches are said to have the same \textbf{swatch type} if they are isomorphic as bipartite, rooted graphs. A \textbf{subswatch} of a swatch is a swatch of smaller radius at the same root vertex.

\begin{figure}
\center
\subfloat[]{%
	\label{free_swatch}{%
		\includegraphics[height=5cm]{%
			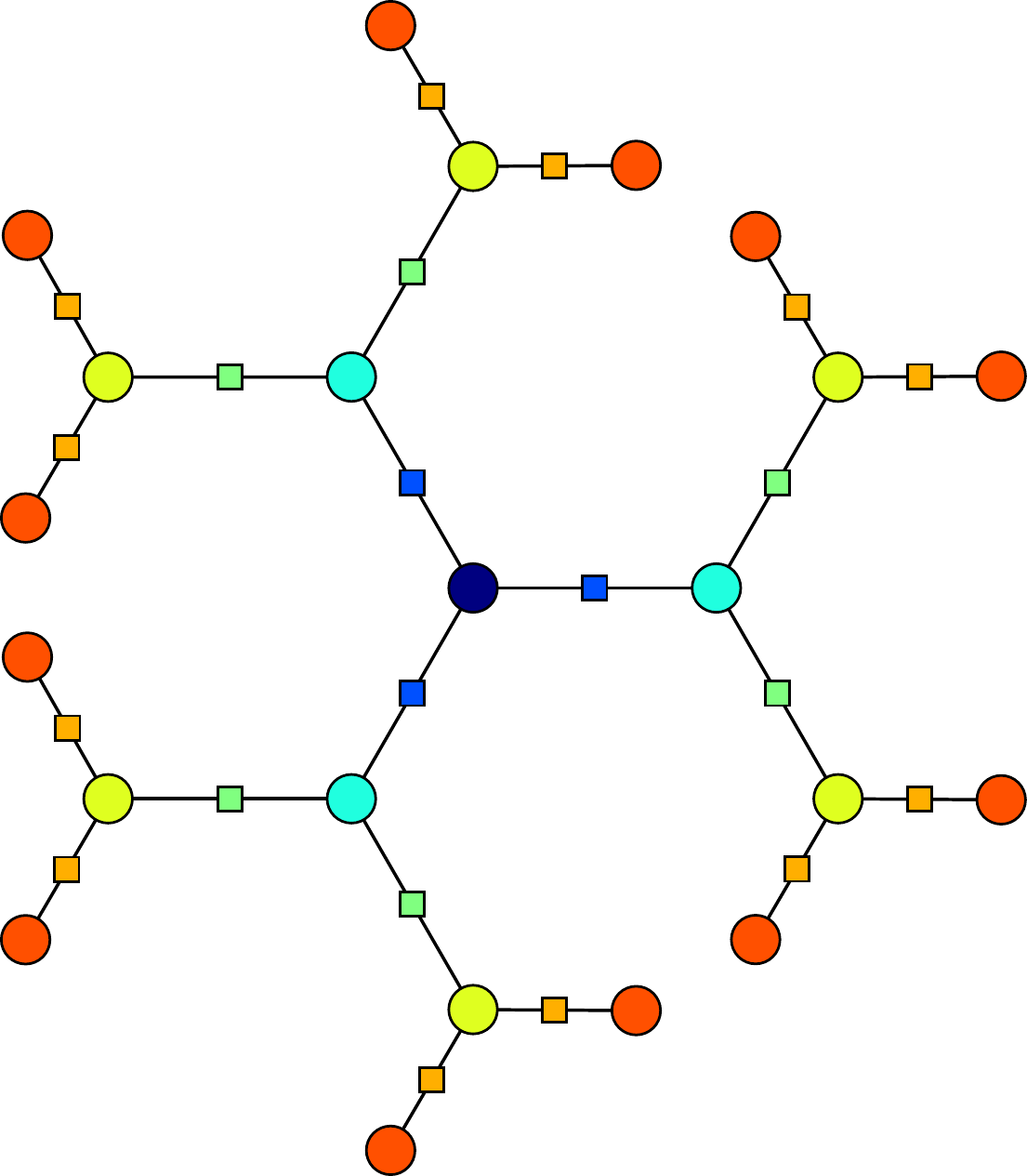}}}
\hspace{20pt}
\subfloat[]{%
	\label{not_free_swatch}{%
		\includegraphics[height=4.75cm]{%
			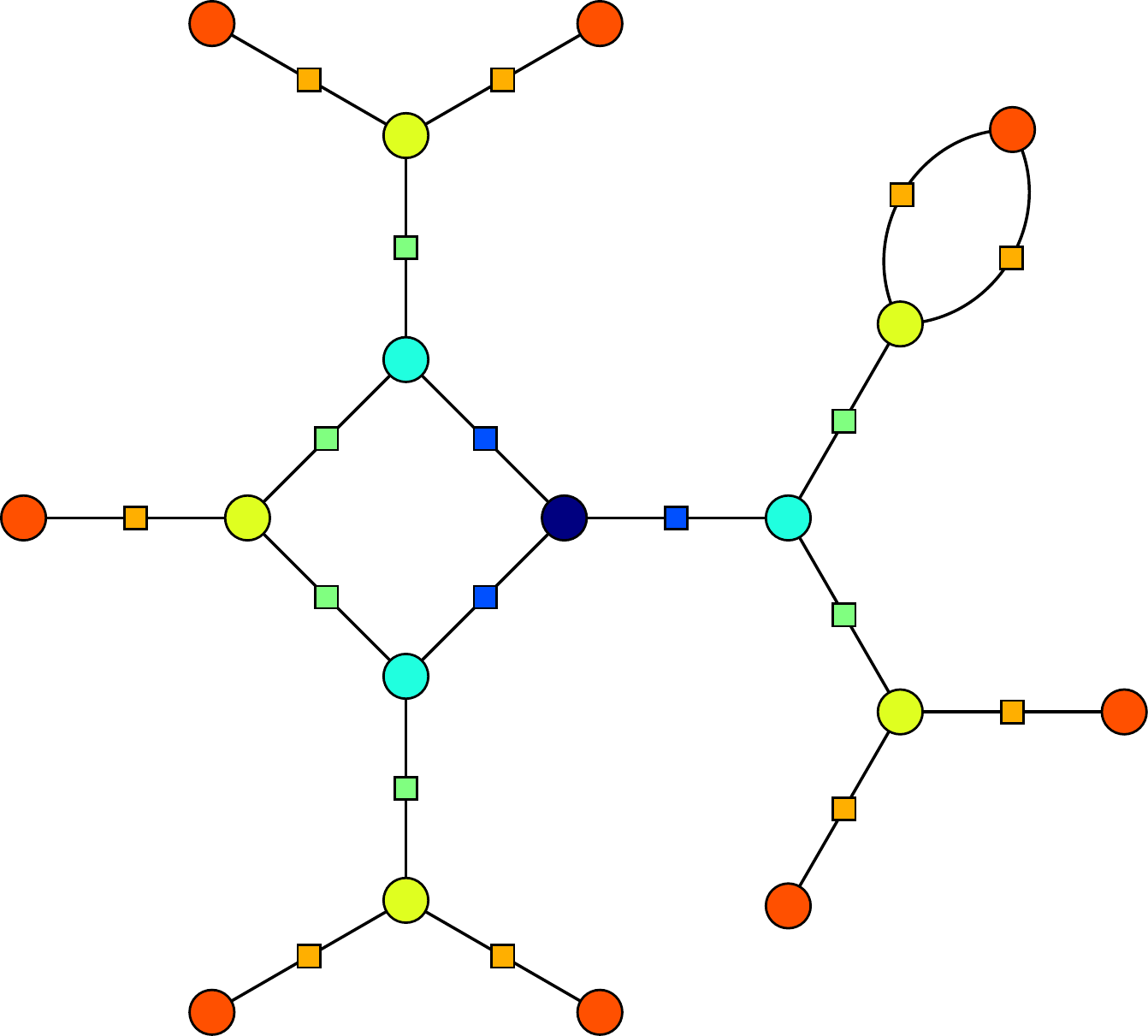}}}
\caption{\label{fig_swatchTypes}Swatch types of radius six in an embedded graph. The vertex color indicates distance from the root, with the root colored dark blue. (a) A free swatch where there is a single path from the root to any given vertex. (b) A swatch that contains a cycle of length four and a cycle of length two.}
\end{figure}

It will be useful later to have a distance on swatches. Let the largest common subswatch of two swatches be the swatch of largest radius that is a subswatch of both. The distance between two swatches is defined as the reciprocal of the number of vertices and edges in the largest common subswatch, or zero if the swatches are the same. For example, the largest radius for which the swatch in Figure~\ref{not_free_swatch} is free is $r = 3$, and the distance to the free swatch in Figure~\ref{free_swatch} is $1/13$.

\begin{Definition}
Let $G$ be a finite, connected graph. The \textbf{topological cloth at radius $r$}, denoted $\mathcal{S}_r\paren{G}$ is the probability measure on the set of swatch types of radius $r$ induced by the counting measure on the vertices. That is, for a swatch type $S,$  $\mathcal{S}_r\paren{G}\paren{S}$ is the proportion of vertices of $G$ whose swatch of radius $r$ is combinatorially equivalent to $S.$
\end{Definition}Letting $r$ vary, we get a family of probability distributions called the topological cloth of $G.$  The topological cloth characterizes the local topology of the graph in the sense of determining the probability at which any local configuration appears in the graph, as well as of prescribing all of its local topological properties (in a sense defined below).

\subsection{A Topological Distance on Graphs}
\label{sec_metric}

We use the earth mover's distance on the topological cloth at radius $r$ to define a family of distances on graphs. Suppose $\paren{X,\rho}$ is a finite metric space, and $P_1$ and $P_2$ are two probability distributions on it. A matching $M$ of $P_1$ and $P_2$ is a probability distribution on $X\times X$ with marginals are $P_1$ and $P_2.$ The earth mover's distance $d_{\text{em}}$, or first Wasserstein metric, is the minimum cost of a matching between $P_1$ and $P_2$ \cite{1781Monge,1998Rubner}:
\s{d_\text{em}\paren{P_1,P_2}=\min_{M}\sum_{x_1,x_2\in X}\rho\paren{x_1,x_2}M\paren{x_1,x_2}}
Given two graphs $G_1$ and $G_2,$ define the distance at radius $r$ to be the earth mover's distance between their topological cloths of radius $r$:
\s{d_r\paren{G_1,G_2}=d_{\text{em}}\paren{\mathcal{S}_r\paren{G_1},\mathcal{S}_r\paren{G_2}}}
$d_r$ is uniformly bounded and non-decreasing in $r$, and it stabilizes for some finite $r$ if the graphs are finite. The limit distance on graphs is defined as the limit of $d_r$ with increasing $r$, or
\begin{equation}
\label{eq_clothDist}
d_\infty\paren{G_1,G_2}=\lim_{r\rightarrow\infty}d_r\paren{G_1,G_2}
\end{equation}Note that a a sequence $\set{G_i}$ is Cauchy in $d_\infty$ if and only if all swatch frequencies converge.

\subsection{Local Topological Convergence}
Let $\mathfrak{G}^\bullet$ be the space of countable, connected graphs with a root vertex specified. The theory of Benjamini-Schramm graph limits associates to a Cauchy sequence in $d_\infty$ a probability distribution on $\mathfrak{G}^\bullet:$

\label{sec_convergence}
\begin{theorem}[Benjamini-Schramm Convergence~\cite{2012lovasz,2001benjamini}]
\label{thm_BS}
Let $G_1,G_2,\ldots=\left\{G_i\right\}$ be a Cauchy sequence of finite graphs in $d_\infty.$ There exists a unique probability distribution $\sigma$ on $\mathfrak{G}^\bullet$ such that $\left\{G_i\right\}$ converges to $\sigma$ in $d.$ The value of $\sigma$ on a basis set $E_S$ is given by $\sigma\paren{E_S}=\sigma_r\paren{S}$ where $r$ is the radius of the swatch $S$ and $\sigma_r$ is the limiting distribution on the set of swatches of radius $r.$ $\sigma$ is called the \textbf{Benjamini-Schramm limit} or \textbf{local topological limit} of the sequence of graphs.
\end{theorem}

\subsection{Consequences of Topological Convergence}
If $G$ is a finite graph, and  $H$ is a finite, two-colored graph graph, let $\text{hom}\paren{H,G}$ be the number of graph homomorphisms $H$ into the adjacency graph on the union of the vertex and edge sets of $G,$ and let $\text{inj}\paren{H,G}$ be the number that are injective. Similarly, $H'$ and $G'$ are (possibly infinite) rooted graphs, let $\text{hom}\paren{H',G'}$ and $\text{inj}\paren{H',G'}$ be the homomorphisms that map root to root.

\begin{Definition}
If $H$ is a finite, two-colored graph and $G$ is a finite graph, the \textbf{injective homomorphism frequency} of $H$ in $G$ is
\s{rho\paren{H,G}=\frac{\text{inj}\paren{H,G}}{\abs{G}}=\sum_{v\in G}\frac{\text{inj}\paren{H',G^v}}{\abs{G}}}
where $G^v$ is $G$ with root vertex $v$ specified.If $\sigma$ is a probability distribution on $\mathcal{G}^\bullet,$ and $H'$ is a rooted graph, the injective homomorphism frequency of $H'$ in $\sigma$ is
\s{\rho\paren{H',\sigma}=\mathbb{E}_{G'\in G}\text{inj}\paren{H',G'}}
where the expectation is taken with respect to $\sigma.$
\end{Definition}The probability distributions occurring as Benjamini-Schramm graph limits have a property called involution invariance, which implies that $\rho\paren{H',\sigma}$ does not depend on the choice of root of $H.$

\begin{Definition}
\label{local_topo_prop}
A \textbf{local topological property} of graphs is any property that can be expressed in terms of a finite combination of injective homomorphism frequencies.
\end{Definition}For example, any quantity of the form $\frac{\text{hom}\paren{H,G}}{\abs{G}}$ is a local topological property~\cite{2012lovasz}.

\begin{theorem}
A sequence of graphs $\set{G_i}$ converges in the local topological sense to a limit distribution $\sigma$ if and only if all local topological properties to the corresponding quantities defined for the Benjamini-Schramm limit.~\cite{2012lovasz}
\end{theorem}

\subsection{Cloths and Limits of Embedded Cell Complexes}
\label{sec_infiniteCloth}
Here, we extend the notions of topological cloth and local topological convergence from finite graphs to to countable, locally finite graphs embedded  in $\R^n.$ Let $G$ be an embedded graph in $\R^n.$ We associate to $G$ to sequences of finite graphs: $I_r\paren{G}$ is the union the vertices and edges of $G$ completely contained in the ball of radius $r$ centered at the origin, and $O_r\paren{G}$ is the union of the vertices and edges of $G$ intersecting that ball. 

\begin{Definition}If $\lim_{r\rightarrow\infty}O_r\paren{G}$ is locally topologically convergent, we say that $G$ has an \textbf{outer-regular topological cloth}. If $I_r\paren{G}$ does, $G$ has an \textbf{inner-regular topological cloth}. If they are equal, $G$ has a \textbf{well-defined topological cloth}, which we call $\mathcal{S}\paren{G}.$ 
\end{Definition}
Note that $\mathcal{S}\paren{G}$ does not depend on the choice of origin in $\R^n.$ If $\set{G_i}$ is a sequence of embedded graphs with well-defined topological cloths, we say that $\set{G_i}$ is locally topologically convergent if the probability distributions $\mathcal{S}\paren{G_i}$ converge strongly to a limit distribution.

\section{Topologies on the Space of Embedded Graphs}
\label{sec_graphTop}

In this section, we will study the properties of the several topologies on the space of embedded graphs, as preparation for the definition of local geometric convergence. The varifold topology is the most natural topology in the context of curvature flow on embedded graphs in $\R^n,$ but the topology induced by the local Hausdorff metric is perhaps more natural for computational applications. In this section, we study the properties of these two topologies, and relate them to a third topology called the smooth topology of topological types. 

Many of the proofs of the propositions in this section are quite technical, and are deferred to an appendix at the end of the paper.

\subsection{The Varifold Topology}

The varifold topology is the vague (or weak*) topology on certain Radon measures associated to embedded graphs. If $X$ is a locally compact Hausdorff space, let $\mathcal{M}^+\paren{X}$ be the space of positive, locally finite Radon measures on $X$. Give $\mathcal{M}^+\paren{X}$ the initial topology with respect to the integrals of compactly, supported, continuous functions. That is, a sequence of Radon measures $\mu_i$ converges to $\mu$ if and only if for every continuous, compactly supported function $f:X\rightarrow\R$ 
\s{\lim_{i\rightarrow\infty}\int_{X}f\;d\mu_i\rightarrow \int_{X}f\;d\mu}
$\mathcal{M}^+\paren{X}$ is metrizable, and if $Y\subset X$ the restriction map $\mathcal{M}^+\paren{X}\rightarrow\mathcal{M}^+\paren{Y}$ is continuous. If $X$ is compact, the vague topology~\cite{2002kallenberg,1976dieudonne}.

If $e$ is an edge of a  graph $G,$ we say that $e\in G$ and denote the tangent bundle of $e$ as $T_x\paren{e}.$ The tangent bundle $T\paren{G}\subset \R^n\times\mathbb{RP}^{n-1}$ of an embedded graph $G$ is
\s{T\paren{G}=\set{\paren{x,v:x\in e \in G, v\in T_x\paren{e}}}}
That is, the set of pairs $\paren{x,v}$ where $x$ is a point of $G$ and $v$ is a unit tangent vector of an edge containing $x.$ The projection map $T\paren{G}\rightarrow G$ is injective away from vertices.

An embedded graph with multiplicity $m$ is a graph $G$ together with a function $m:G\rightarrow \N$ that is constant on each edge. To such a graph, we associate a positive Radon measure $\mu_G$ on $\R^n\times \mathbb{RP}^{n-1}$ that assigns to a Borel subset of $Y$ of $\R^n$
\s{\mu_G\paren{Y}=\int_{T\paren{G\cap X}}m\;d\mathcal{H}^1}
where $\mathcal{H}^1$ is the first-dimensional Hausdorff measure. The \textbf{varifold topology} on the space of embedded graphs is the vague topology on the measures $\mu_G.$ associated measures. In other words a sequence of graphs $\set{G_i}$ converges to $G$ if and only if for any continuous, compactly supported function $f:\R^n\times\mathbb{RP}^{n-1}\rightarrow\R,$ 
\s{\lim_{i\rightarrow\infty} \int_{\R^n}f\mu_{G_i} = \lim_{i\rightarrow\infty} \int_{T\paren{G_i}}m_i f\;d\mathcal{H}^1\rightarrow \int_{T\paren{G}}m f\;d\mathcal{H}^1= \int_{\R^n}f\mu_{G}}
\begin{Definition}
Let $\check{\mathcal{G}}^n_r$ be the space of embedded graphs with multiplicity in the open $n$-ball of radius $r$ with the varifold topology. Let $\mathcal{G}^n$ be the space of embedded graphs with multiplicity in $\R^n$ with the same topology.
\end{Definition}For more information on the properties of the varifold topology, we suggest~\cite{1983simon,deLellis,1972allard}.
   
\subsection{The Hausdorff Metric Topology}
If $X$ and $Y$ are subsets of $\R^n$ the Hausdorff distance between them is 
\s{d_H\paren{X,Y}=\max\Big(\sup_{x\in X}\inf_{y\in Y} d\paren{x,y},\sup_{y\in Y}\inf_{x\in X} d\paren{x,y}\Big).}That is, the smallest $\epsilon$ so that $Y$ is contained in the closed $\epsilon$ neighborhood of $X$ and visa versa. $d_H$ makes the set of compact subsets of of any bounded subset of $\R^n$ into a complete metric space.  

Let $B_r$ be the open ball of radius $r$ centered at the origin in $\R^n,$ and $K\paren{B_r}$ be the space of relatively compact subsets of $B_r$ with the Hausdorff metric topology. If $r<s,$ the map from $K\paren{B_s}$ to $K\paren{B_s}$ given by intersecting a subset of $B_s$ with $B_r$ is not continuous. We introduce a different metric inducing at topology for which this map is continuous: 

\begin{Definition}
If $X_1$ and $X_2$ are subsets of the open ball of radius $r,$ $B_r,$ the \textbf{local Hausdorff distance} between them is
\s{d_L\paren{X_1,X_2}=d_H\paren{X_1\cup \partial B_r,X_2\cup \partial B_r}}
\end{Definition}

\begin{Definition}
$\mathcal{G}^n_r$ is the space of embedded graphs in the open ball of radius $r$ in $\R^n$ with the topology induced by the local Hausdorff distance. Let $p_{s,r}:\mathcal{G}_s^n\rightarrow \mathcal{G}_r^n$ for $s\geq r$ be the continuous map given by intersecting an embedded graph in $B_s$ with $B_r.$ Also, define $\mathcal{G}^n$ to be the inverse limit of the spaces $\mathcal{G}_r^n$  with respect to this family of maps.
\end{Definition}  

$\mathcal{G}_r^n$ is separable, with a dense subset given by graphs whose vertices have rational coordinates, and whose edges are given by polynomial maps from $\brac{0,1}$ to $B_r.$  $\mathcal{G}^n$ can be equivalently defined as the space of embedded graphs in $\R^n$ with the initial topology with respect to the intersection maps $p_r:\mathcal{G}^n\rightarrow \mathcal{G}_r^n$ for $n\in \N.$

Let $r:\check{\mathcal{G}}_r^n\rightarrow\mathcal{G}_r^n$ be the map induced by forgetting the multiplicity.
\begin{restatable}{proposition}{propOne}
\label{prop_HausdorffVarifold}
Let $\set{\check{G}_i}$ be a sequence of graphs in $\check{\mathcal{G}}^n_r$ that converges to a limit $\check{G}\in\check{\mathcal{G}}^n_r$. Then $\set{r\paren{G_i}}$ converges to $r\paren{\check{G}}.$ That is, $r$ is continuous.
\end{restatable} See Appendix~\ref{appendix} for the proof. A partial converse to this result follows from the Allard compactness theorem~\cite{1972allard}. If $G_i\rightarrow G$ in the local Hausdorff metric, and the $G_i$ have uniformly bounded curvatures and masses, that result implies that there is a convergent subsequence in the varifold topology, and that limit must equal $G$ modulo multiplicity. 

\subsection{The Smooth Topology of Topological Types}
\label{sec_topType}
A \textbf{topological type} $S$ of $\mathcal{G}_{r}^{n}$ is, as a set, equal to the collection of graphs in $\mathcal{G}_{r}^{n}$ sharing a single combinatorial isomorphism type. We will define a topology of smooth convergence on $S$ in which a sequence of graphs converges if and only if their edges converge smoothly to edges.

The non-linear Grassmanian $\text{Gr}\paren{\brac{0,1},\R^n}$ of smoothly embedded curves with boundary is the space
\s{\text{Gr}\paren{\brac{0,1},\R^n}=\frac{\text{Emb}\paren{\brac{0,1},\R^n}}{\text{Diff}\paren{\brac{0,1},\brac{0,1}}}}
where $\text{Emb}{\brac{0,1},\R^n}\subset C^{\infty}\paren{\brac{0,1},\R^n}$ is is the space of embeddings of $\brac{0,1}$ in $\R^n$ and $\text{Diff}\paren{\brac{0,1},\brac{0,1}}$ is the space of diffeomorphisms of the interval. $\text{Gr}\paren{\brac{0,1},n}$ is a smooth Frechet manifold~\cite{2014balmaz}, and can be metrized as follows. $C^\infty\paren{\brac{0,1},\R^n}$  is a Frechet space, there is a metric $\hat{d}$ inducing its topology. For a curve $x\in \text{Gr}\paren{\brac{0,1},\R^n},$ let $\phi_x^1:\brac{0,1}\rightarrow\R^n$ and $\phi_x^2$ the two unit-speed parametrizations. Then
\s{d\paren{x,y}=\min\paren{\hat{d}\paren{\phi_{x}^1,\phi_y^1},\hat{d}\paren{\phi_{x}^2,\phi_y^1}}}
induces the topology on $\text{Gr}\paren{\brac{0,1},n}.$\footnote{We will use metrizability of this space to apply the Portmanteau theorem. A difficulty in generalizing the concept of local geometric convergence to higher dimensional regular cell complexes is showing that the corresponding non-linear Grassmanians are completely regular.} 

If $S$ is a topological type with $l$ edges, there is an injective function $\psi:S\rightarrow\text{Sym}^{l}\paren{\text{Gr}^{D}\paren{k,n}}$ given by sending an embedded graph in $S$ to the tuple of the closures of its edges inside the closed ball of radius $r.$ Topologize $S$ as a subspace of $\text{Sym}^{l}\paren{\text{Gr}^{D}\paren{k,n}},$ so a sequence of graphs converges if and only if each of its edges converges in the smooth topology. $S$ is metrizable and second-countable.

To be precise, let $\set{S^i_{n,r}}$ be the set of topological types of $\mathcal{G}_r^n,$ one for each graph isomorphism class, and let $s_i:S^i_{n,r}\rightarrow \mathcal{G}^{n}_r$ be the map induced by inclusion.

\begin{Definition}
The space of embedded graphs in the ball of radius $r$ in $\R^n$ with the \textbf{smooth topology of topological types}, $\hat{\mathcal{G}}^{n}_r$ is the set $\mathcal{G}^n_r$ with the final topology with respect to the maps $s_i.$
\end{Definition}

\begin{Definition}
\label{def_local_prop}
A \textbf{local geometric property} of an embedded graph in $\R^n$ is a function of the form $f=\hat{f} \circ s^{-1}\circ p_r$ where $\hat{f}:\hat{\mathcal{G}}_r^n\rightarrow \R$ is a bounded, continuous function on $\hat{\mathcal{G}}_r^n.$
\end{Definition}

A \textbf{topological type with multiplicity} of $\check{\mathcal{G}}_r^n$ is as a set, equal to the collection of graphs in $\check{\mathcal{G}}_{r}^{n}$ sharing the same combinatorial isomorphism type as colored graphs, with the multiplicity giving the coloring. The construction of a smooth topology of topological types proceeds identically to that for graphs without multiplicity. In the next section, we introduce a concept of $\delta$-thickenability that allows us to mediate between these smooth topologies and the Hausdorff metric and varifold topologies.

\subsection{Thickenability of Embedded Graphs}
A $\delta$-thickenable graph in $\mathcal{G}_r^n$ is, roughly speaking, one whose combinatorial properties are unchanged if each vertex is replaced by a ball of radius $\epsilon<\delta/2,$ each edge is replaced by its $\epsilon^2/2$-tube, and it is intersected with the ball of radius $r-\epsilon/2.$  It is related to the notion of control data for a stratified space.  First, we require two concepts:
 \begin{itemize}
\item The (open) $\epsilon$-neighborhood of a a subset $X$ of $\R^n$ is the set of all points within distance $\epsilon$ of it. It is denoted $X_\epsilon.$
\item A smooth curve $e$ in $\R^n$ has an (open) tubular neighborhood of radius $\epsilon$ if every point in $e_\epsilon$ is closest to a unique point of $e.$ This neighborhood is denoted $T_\epsilon\paren{e}.$ If $e$ has a tubular neighborhood of radius $\epsilon,$ then the curvature of $e$ is above by $1/\epsilon,$ and that the intersection of $e$ with any ball of radius less than $\epsilon/2$ is connected.
\end{itemize}

\begin{figure}
\center
\subfloat[]{%
	\label{fig_beforeThicken}{%
		\includegraphics[width=.3\textwidth]{%
			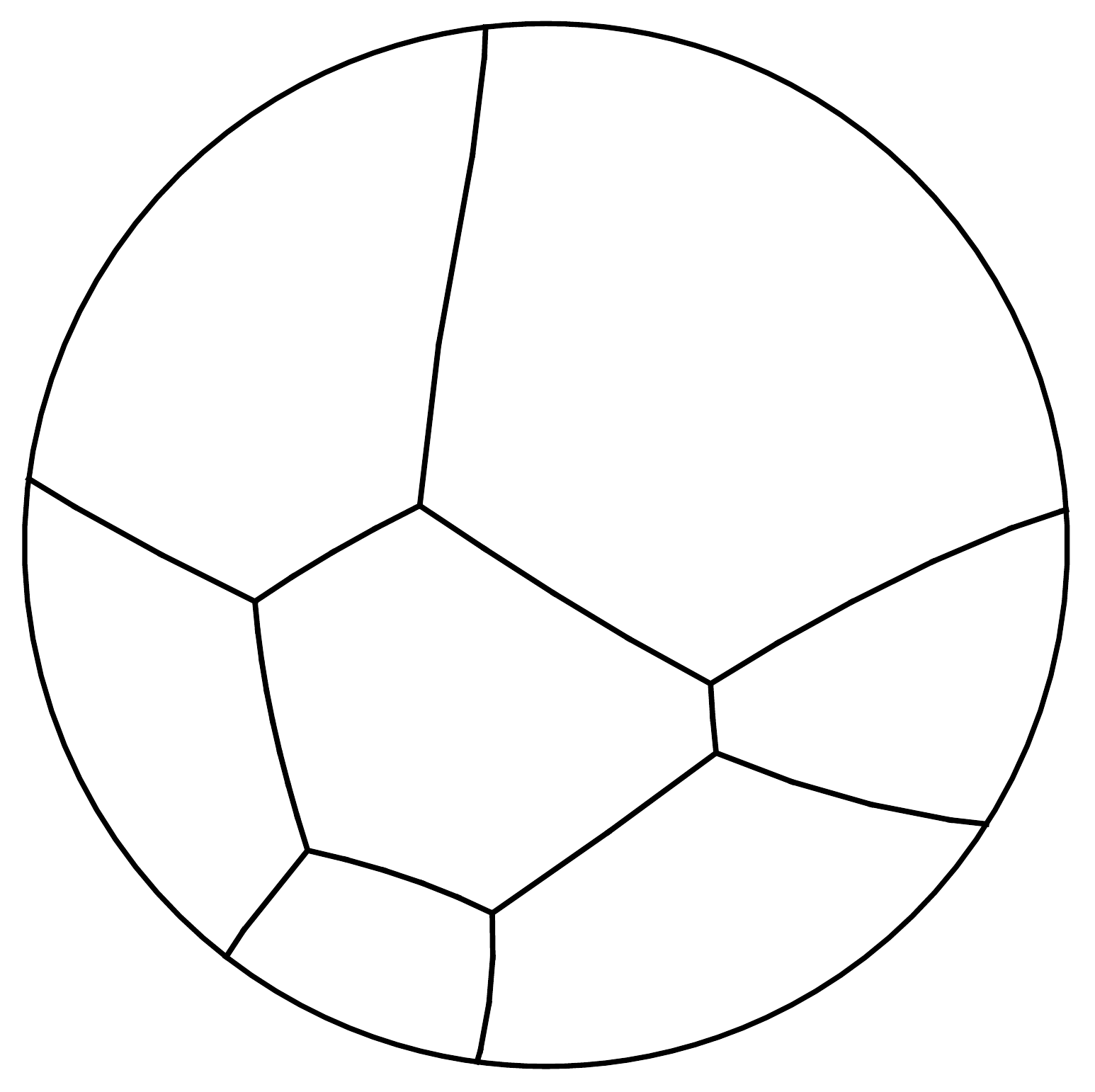}}}
\hspace{20pt}
\subfloat[]{%
	\label{fig_afterThicken}{%
		\includegraphics[width=.3\textwidth]{%
			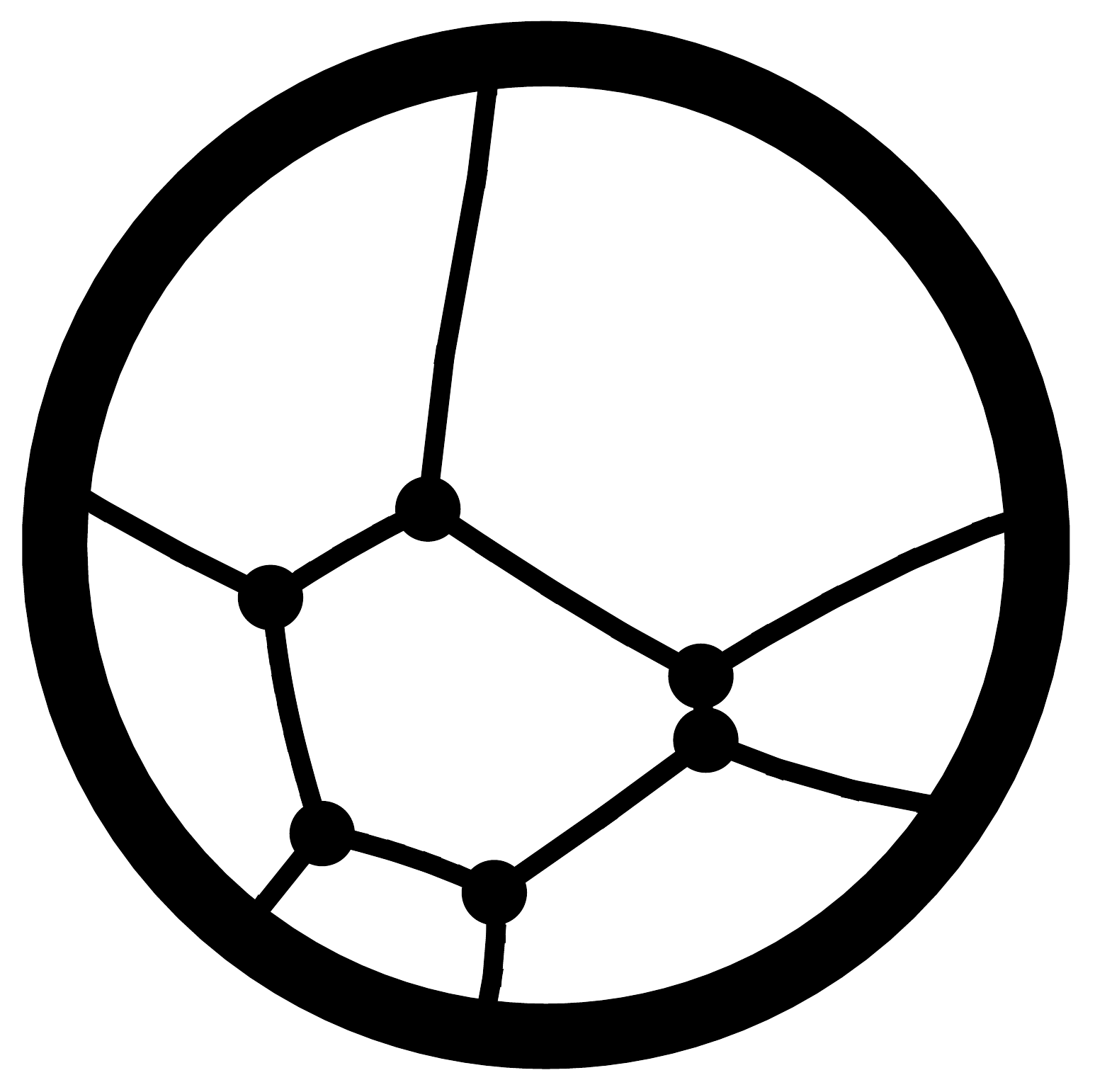}}}
\caption{\label{fig_thicken} (a) An embedded graph in $B_r$ that is $\delta$-thickenable for $\delta\leq \Delta=.133r.$ The limiting factor for thickenability is a pair of vertices of distance $\Delta$ apart. (b) The graph obtained by replacing each vertex with a ball of radius $\Delta/2$ and each edge with a tube of radius $\Delta^2/4.$}
\end{figure}

\begin{Definition}
$G\in\mathcal{G}^n_r$ is \textbf{$\delta$-thickenable} if it satisfies the following for all $\epsilon<\delta$:

\begin{enumerate}
\item If $v$ and $w$ are vertices of $G,$ then $d\paren{v,w}\geq\delta.$ Also, the only edges of $G$ that intersect $B_\delta\paren{v}$ are those adjacent to $v.$
\item Each edge has a tubular neighborhood of radius $2\epsilon.$ 
\item If $e$ is an edge of $G$ adjacent to vertices $v_1$ and $v_2$ then 
\s{T_{\epsilon^2}\paren{e}\cap \paren{G-e}\subset B_{\epsilon}\paren{v_1}\cup B_{\epsilon}\paren{v_2}\cup \paren{B_r-B_{r-\epsilon}}}
\item  Every vertex and edge of $G$ intersects $B_{r-\delta}.$
\item If $e$ is an an edge and $\paren{B_r-B_{r-\epsilon}}\cap e\neq \emptyset,$ then each connected connected component of $\paren{B_r-B_{r-\epsilon}}\cap e$ intersects $\partial B_{r-\epsilon}$ in exactly one point. 
\end{enumerate}Note that if $G$ is $\delta$-thickenable, then it is $\epsilon$-thickenable for all $\epsilon\leq \delta.$ An embedded graph $G\in\mathcal{G}^n$ graph is \textbf{locally thickenable} if its intersection with any open ball is $\delta$-thickenable for some $\delta>0.$ 
\end{Definition} An example of a graph in $B_r$ for which vertex separation is the limiting factor for thickenability is shown in Figure~\ref{fig_beforeThicken}.

\begin{example}
Let $G$ be an embedded graph in $\R^2$ whose vertices are at least distance $D$ apart, and whose edges are straight line segments meeting at angles of at least $\theta.$ Then $G$ is $\min\paren{D,2\sin\paren{\theta/2}}$-thickenable in $\R^2.$
\end{example}

The proof of the following lemma and proposition are included in Appendix~\ref{appendix}.
\begin{restatable}{lemma}{lemmaTwo}
All embedded graphs $G\in\mathcal{G}^n$ are locally thickenable.
\end{restatable}

\begin{restatable}{proposition}{propTwo}
\label{prop2}
Let $X_{\delta,D}\subset \mathcal{G}_r^n$ be the set of $\delta$-thickenable embedded graphs such that the $n-1$ Frenet-Serret curvatures of each edge are bounded above by $D.$
\begin{enumerate}
\item $X_{\delta,D}$ is compact, and equals the disconnected union of sets of constant topological type. 
\item  If $\hat{X}_{\delta,D}\subset \hat{\mathcal{G}}_r^n$ is the same set with the smooth topology, then  $\hat{X}_{\delta,D}$ is homeomorphic to $X_{\delta,D}.$
\end{enumerate}
The corresponding statements are also true for embedded graphs with multiplicity in $\check{\mathcal{G}}_r^n.$ 
\end{restatable}

\section{Local Geometric Convergence}
\label{sec_geoConv}

In this section, we propose a notion of convergence for embedded graphs that implies the convergence of averages of a large class of local geometric properties. This is called \textbf{local geometric convergence}, and was developed in analogy with Benjamini-Schramm convergence. In short, we associate to an embedded graph $G\in\mathcal{G}^n$ an empirical probability distribution $\mathcal{P}\paren{G}$ on $\mathcal{G}^n$, called the the \textbf{geometric cloth}. Local geometric convergence is related to weak convergence of geometric cloths.

The definitions of local geometric convergence for embedded graphs in $\mathcal{G}^n$ and embedded graphs with multiplicity in $\check{\mathcal{G}_r^n}$ are nearly identical, as are the proofs of their properties. For brevity, we will  work with $\mathcal{G}^n$ here.

\subsection{Preliminaries: Weak Convergence and the Ergodicity}
\label{sec_prelim}
We will require the following definitions and results.

\subsubsection{Weak Convergence}

\begin{definition}
A sequence of probability measures $\set{\mu_i}$ on a completely regular space $X$ \textbf{converges weakly} to a probability measure $\mu$ if
\s{\lim_{i\rightarrow\infty}\int_{X}f\;d\mu_i=\int_{X}f\;d\mu}
for all bounded continuous functions $f:X\rightarrow \R.$   
\end{definition}

\begin{theorem}[Portmanteau theorem~\cite{1970topsoe}]
Let $X$ be a completely regular topological space, $\mu$ a probability measure, and $\mu_i$ a sequence of probability measures on $Y.$ The following are equivalent:
\begin{enumerate}
\item $\mu_i\rightarrow \mu$ weakly.
\item $\limsup\mu_i\paren{C}\leq \mu\paren{C}$ for all closed subsets $C\subset X.$
\item $\liminf\mu_i\paren{U}\geq\mu\paren{U}$ for all open subsets $U\paren X.$
\item $\mu_i\paren{Z}\rightarrow \mu_i\paren{Z}$ for all $Z\subset X$ with $\mu\paren{\partial Z}=0.$
\end{enumerate}
\end{theorem}

Note that all metric spaces are completely regular, and that the hypotheses do not require $\mu$ to be a Radon measure.

\begin{Definition}
A collection of probability measures $M$ on a space $X$ is \textbf{tight} if for all $\epsilon>0$ there is a compact set $K_{\epsilon}$ so that for all $\mu\in M$
\s{\mu\paren{K_{\epsilon}}>1-\epsilon}
\end{Definition}

\begin{theorem}(Prokhorov's Theorem)
The closure of a tight collection of probability measures on a separable metric space is compact in the weak topology.
\end{theorem}

\subsection{Uniformly Separating  Sequences}
\label{sec_USS}
Let $s:\hat{\mathcal{G}}^{r}_n\rightarrow\mathcal{G}^{r}_n$ be the identity map. It is bijective and continuous, and is a homeomorphism when restricted to the set $\hat{X}_m$ of $1/m$-thickenable graphs whose Frenet-Serret curvatures bounded above in magnitude by $m.$ Let $X_m=s\paren{\hat{X}_m}.$ 

\begin{lemma}
$s^{-1}$ is a Borel function.
\end{lemma}
\begin{proof}

If $B$ be a Borel set of $S$ then
\s{s\paren{B}=s\paren{\bigcup_{m\in\N}{\hat{X}_m\cap B}}=\bigcup_{m\in\N}s\paren{\hat{X}_m\cap B}=\bigcup_{m\in\N}X_m\cap s\paren{\hat{X_m}\cap B}}
which is a Borel set, because the countable union and homeomorphic image of Borel sets are Borel.
\end{proof}The previous lemma implies that Borel measures on $\mathcal{G}^{r}_n$ pull back to Borel measures on $\hat{\mathcal{G}}^{r}_n,$ so we can make the following definition:  
\begin{Definition}
Let $\mu$ be a probability measure on $\mathcal{G}^n_r.$ The \textbf{induced smooth measure} $s^*\mu$ on $\hat{\mathcal{G}}^{r}_n$ is the pushforward of $\mu$ by $s^{-1}.$
\end{Definition}

\begin{Definition}
\label{defn_unifSep}
A sequence of probability measures $\set{\mu_i}$ on  $\mathcal{G}_{r}^n$ is \textbf{uniformly separating} for all $\epsilon>0$ there is an $M$ so that each $\mu_i$ satisfies
\s{\mu_i\paren{X_M}>1-\epsilon}
Similarly, a sequence of probability measures $\set{\nu_i}$ on $\mathcal{G}^n$ is uniformly separating if $\set{{p_r}_*\nu_i}$ is for each $r\in\N.$ 
\end{Definition} $X_M$ is compact, so a uniformly separating sequence on $\mathcal{G}_{r}^n$  is tight.
\begin{proposition}
\label{unifProp}
Let $\set{\mu_i}$ be a uniformly separating sequence of probability measures on $\mathcal{G}_{r}^n.$ $\mu_i$ converges weakly to $\mu$ if and only if $s^*\mu_i$ converge weakly to $s^*\mu.$
\end{proposition}
\begin{proof}
The smooth topology is finer than that of $\mathcal{G}_{r}^n,$ so weak convergence of  $s^*\mu_i$ to $s^*\mu$ implies weak convergence of $\mu_i$ to $\mu,$ even without the requirement of uniform separation. 

For the other direction, assume that $\mu_i\rightarrow \mu$ weakly and let $U$ be an open set of $\hat{\mathcal{G}}^{r}_n.$ We would like to show that 
\s{\liminf_i \mu_i\paren{s\paren{U}}\geq \mu\paren{s\paren{U}}} Let $\epsilon>0.$ $\set{\mu_i}$ is uniformly separating, so we can find an $M$ so that $\mu_i\paren{X_M}>1-\epsilon$ for all $i.$ $X_M$ is closed so $\mu\paren{X_M}>1-\epsilon,$ as well. $s$ is a homeomorphism when restricted to $\hat{X}_{M},$ so $s\paren{U}\cap X_M=s\paren{U\cap\hat{X}_{M}}$ is open in $X_{M}.$ It follows that there is an open set $V$ of  $\mathcal{G}_{r}^n$ with $X_M\cap s\paren{U}=X_{M}\cap V.$ Choose $I$ so that $\mu_i\paren{V}>\mu\paren{V}-\epsilon$ for all $i>I.$ Then
\begin{align*}
\mu_i\paren{s\paren{U}}&\\
&\geq \mu_i\paren{X_{M}\cap s\paren{U}}-\epsilon\\
&= \mu_i\paren{X_M\cap V}-\epsilon\\
&\geq \mu_i\paren{V}-2\epsilon\\
&\geq\mu\paren{V}-3\epsilon\\
& \geq \mu\paren{X_{M}\cap s\paren{U}}-4\epsilon\\
& \geq \mu\paren{s\paren{U}}-5\epsilon
\end{align*}
So 
\s{\liminf_i \mu_i\paren{s\paren{U}}\geq \mu\paren{s\paren{U}}}
and $s^*\mu_i\rightarrow s^*\mu$ weakly, as desired. 
\end{proof}

\subsection{The Geometric Cloth}
\label{sec_geoCloth}
 $\R^n$ acts on $\mathcal{G}^{n}$ by translation. Let $G\in\mathcal{G}^n$ be an embedded graph and let $\phi^G:\R^n\rightarrow\mathcal{G}^{n}$ be the map sending $x$ to the graph $xG.$ Also, let $\phi_r^G:\R^n\rightarrow\mathcal{G}_r^n$ be the composition of $p_r\circ\phi^G,$ as shown in Figure~\ref{fig:geometric_cloth}.

\begin{figure}
\center
\includegraphics[width=.6\textwidth]{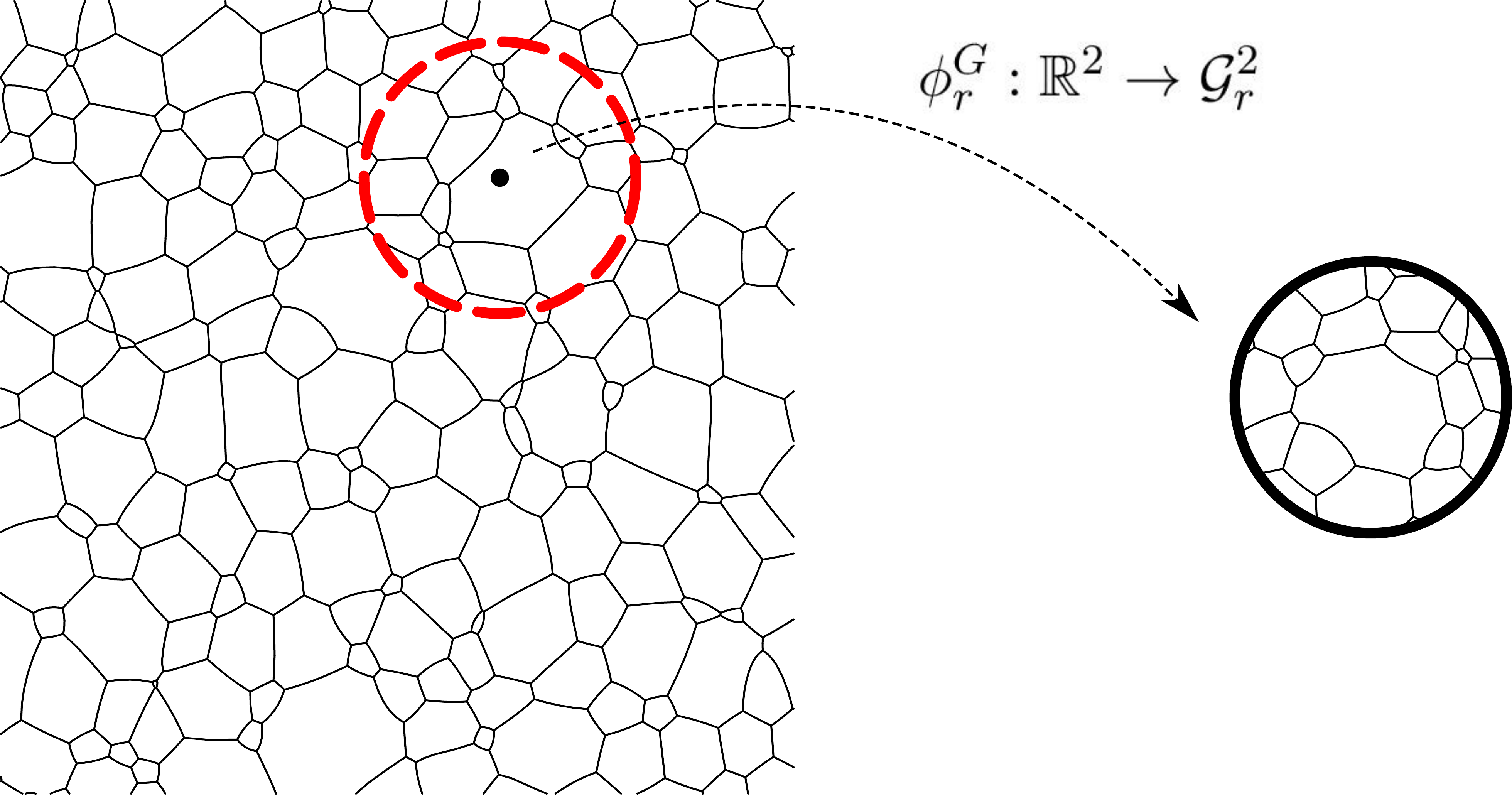}
 \caption{\label{fig:geometric_cloth} An embedded graph $G\in\mathcal{G}^2$ induces a ``window map'' $\phi_r^G:\R^2\rightarrow\mathcal{G}_r^2.$}
\end{figure} 

\begin{Definition}
\label{def_geo_cloth}
Let $G\in\mathcal{G}^{n}.$ If $K$ is a bounded subset of $\R^n,$ define the \textbf{geometric cloth of $G$ in K}, let $P\paren{G,K}$ be the pushforward of the normalized Lebesgue measure on $K$ via $\psi_G.$ If $G$ is infinite, and the sequence $\set{P\paren{G,B_i}}_{i\in\N}$ converges weakly to a non-zero limit distribution $\mathcal{P}\paren{G},$ we call $\mathcal{P}\paren{G}$ the \textbf{geometric cloth} of $G.$ Otherwise, if $G$ is finite, we refer to the geometric cloth of $G$ in its convex hull its geometric cloth $\mathcal{P}\paren{G}.$
\end{Definition} That is, $P\paren{G,K}$ is the probability distribution of graphs embedded in balls of radius $r$ centered at points of $K.$ Note that the definition of geometric cloth, and none of the following definitions in this section, depend on the choice of origin in $\R^n$ because the percentage overlap between the balls of radius $r$ centered at any two points in $\R^n$ goes to $1$ as $r\rightarrow\infty.$

 We are interested in infinite graphs because the conjectures about curvature flow on graphs in Section~\ref{sec_conj} are most natural in that context. An infinite graph has smooth geometric cloth if its local geometric properties have well-defined averages over balls:
\begin{Definition}
If $G$ is infinite and the induced smooth measures $\hat{P}_r\paren{G,B_i}=s^*{p_r}_*P\paren{G,B_i}$ on $\hat{\mathcal{G}}_r^n$ also converge weakly to limit distributions, we say that $G$ has a \textbf{smooth geometric cloth} (for example, if $P\paren{G,B_i}$ is a uniformly separating sequence). Finite graphs always have smooth geometric cloth.
\end{Definition}

An equivalent and perhaps more intuitive way to define the geometric cloth is via local probability distributions:
\begin{proposition}
\label{localProp}
An infinite embedded graph $G$ has a geometric cloth if and only if the probability distributions  $P_r\paren{G,B_i}$ converge weakly to non-zero limit distributions $\mathcal{P}_r\paren{G}$ as $i\rightarrow\infty$ for all $r\in\N.$ 
\end{proposition}
\begin{proof}
Recall that the topology on $\mathcal{G}^{n}$ is the initial topology with respect to the maps $p_r:\mathcal{G}^{n}\rightarrow \mathcal{G}_r^{n}$ for $r\in\N.$ For $r\leq s,$ let $p_{s,r}:\mathcal{G}_s^{n}\rightarrow \mathcal{G}_r^{n}$ be the natural map given by intersecting a graph in $B_s$ with $B_r.$ $\mathcal{G}^{n}$ is the inverse limit with respect to the maps $p_{s,r}.$

Weak convergence of $\mathcal{P}\paren{G,B_i}$ to a limit clearly implies weak convergence of the local distributions $\mathcal{P}_r\paren{G,B_i}.$

Conversely, suppose that $P_r\paren{G,B_i}$ converges weakly to $\mu^r$ for each $r\in\N.$ Note that $\mu_r={p_{s,r}}_*\paren{\mu_s}$ for any $s>r.$ Let $\mathcal{B}\paren{\mathcal{G}_r^n}$ be the Borel $\sigma$-algebra of $\mathcal{G}_r^n.$ The collection of sets 
\s{\mathcal{A}=\bigcup_{r\in\N}\set{p_r^{-1}\paren{B}: B\in \mathcal{B}\paren{\mathcal{G}_r^{n}}}}
is a semi-ring of Borel sets of $\mathcal{G}^{n}$ generating its topology. Define a pre-measure $\nu_0$ on $\mathcal{A}$ by $\nu_0\paren{B}=\mu^r\paren{p_r\paren{B}}$ if $B=p_r^{-1}\paren{B_0}$ for some $B_0\in\mathcal{B}.$ The consistency of the measures $\mu^r$ and $\mu^s$ implies that this is well-defined. By the Caratheodory extension theorem, $\nu_0$ extends to a unique probability measure $\mu$ on $\mathcal{G}^{n}.$ Weak convergence in a $\pi$-system of sets implies weak convergence with respect to the $\sigma$-algebra it generates~\cite{1999billingsley}, so $\set{P\paren{G,B_i}}$ converges weakly to $\mu.$ 
\end{proof}
 
The importance of the geometric cloth is that its local geometric properties have well-defined averages over balls in the following sense:

\begin{Definition}
Let $G\in\mathcal{G}^n$ be an embedded graph, $f$ be a local geometric property, and let $\set{F_i}$ be an increasing sequence of compact sets whose union is $R^n.$ $f$ has a \textbf{well-defined average} $\bar{f}$ over $\set{F_i}$ if
\s{\bar{f}=\lim_{i\rightarrow\infty}\frac{1}{\abs{F_i}}\int_{F_i}f\circ \psi_G\;d\mu}
where $\mu$ is the standard Lebesgue measure on $\R^n.$ If the average of $f$ over the sequence of balls centered at a point in $\R^n$ exists, is simply referred to as the \textbf{well-defined average} of $f$ for $G,$ $\widebar{f_G}.$
\end{Definition} If $G$ has a smooth geometric cloth, and $f$ is a local geometric property, then $f$ has a well-defined $\widebar{f_G}$ and
\s{\widebar{f_G}=\int_{\R^n}f\;d\mathcal{P}\paren{G}}. The converse is true for graphs with smooth geometric cloth for which $\set{P_r\paren{G,B_i}}$ is uniformly separating:
\begin{proposition}
An embedded graph $G\in\mathcal{G}^n$ for which $P\paren{G,B_i}$ is a uniformly separating sequence has a smooth geometric cloth if and only if the averages of local geometric properties over the sequence $\set{B_r}$ converge as $r\rightarrow\infty.$
\end{proposition}
\begin{proof}
The uniform separation property implies that the sequence $\set{P\paren{G,B_r}}$ is tight, and has a convergent subsequence $\mu.$ The convergence of averages of all local geometric properties implies that the limit of any convergent subsequence must be $\mu,$ so $P\paren{G,B_r}\rightarrow \mu$ and $\mu=\mathcal{P}\paren{G}.$ 
\end{proof} We will examine hypotheses under which the local properties of graphs with smooth geometric cloth have averages over more general sequences in~\ref{sec_hom}.

Local geometric properties include a wide range of important properties, from subgraph densities to the percentage of edges with curvature above a certain value. In contrast, the set of continuous functions on $\mathcal{G}^r_n$ or $\check{\mathcal{G}}_r^n$ is relatively depaupaerate and, for example, does not include most topological properties. That is why weak convergence of probabilities measures on those spaces alone is not sufficient for our purposes. 

However, the restriction to bounded functions does exclude many important properties. For example, the number of vertices contained in a ball of radius $r$ in $\R^n$ is continuous with respect to $\hat{\mathcal{G}}_r^n,$ but is unbounded. To obtain more local geometric properties, one can restrict consideration to a smaller class of embedded graphs, such as graphs whose edge length is uniformly bounded.

In summary, embedded graphs with smooth geometric cloth have local geometric properties which have well-defined averages over balls, and those averages are captured by the cloth. 
\subsection{Examples}
\label{sec_geoClothExamples}
\subsubsection{Periodic Graphs}

A periodic embedded graph $G$ is one that is invariant under the translation action of a sublattice $\Z^n\subset\R^n.$ Such a graph has a unit cell $T\subset\R^n$ of least volume so that
\s{G=\bigcup_{t\in\Z^n}t\paren{G\cap T}}
The the geometric cloth of $G$ equals empirical probability distribution $P\paren{G, T}.$

\subsubsection{Poisson Point Process}
Let $\mathcal{C}^{0,n}$ and $\mathcal{C}_r^{0,n}$ and be the spaces of locally finite point collections in $\R^n$ and $B_r,$ respectively. A topological type $S_i$ of $\mathcal{C}_r^{0,n}$ is the subset consisting of all point collections with $i$ points. There is a natural map $\phi_i:S_i\rightarrow \text{Sym}^{i}B_r,$ and the Lebesgue measure on $B_r$ induces a probability measure $\nu_i$ on $S_i$ that selects $i$ points in the ball uniformly and independently.

The Poisson Point Process $\mu_\lambda$ is defined via marginal distributions $\mu_{\lambda,r}$ on $\mathcal{C}_r^{0,n}.$ $\mu_{\lambda,r}$ assigns each set $S_i$ probability $\frac{\lambda^n}{n!},$ and its condition distribution on each $S_i$ is $\nu_i.$

\subsubsection{Voronoi Graphs}

\label{ex_Voronoi}
Let $:V^n:\mathcal{C}^{0,n}\rightarrow\mathcal{G}^{n}$  be the function sending a locally finite point collection to the one-skeleton of the associated Voronoi tessellation. The probability distribution of Poisson-Voronoi graphs in $\R^n$ with intensity $\lambda$ is
\s{\mathcal{V}_\lambda^n=V^n_{*}\paren{\mu_\lambda}} The translation action of $\R^n$ is ergodic with respect to $\mathcal{V}_\lambda^n,$ so for $\mathcal{V}_\lambda^n$-almost all $G,$
\s{\mathcal{P}\paren{G}=\mathcal{V}_\lambda^n} $\mathcal{P}\paren{G}$ assigns probability zero to the orbit of $G$ under the $\R^n$ action. An example of a Voronoi graph in the plane is shown in Figure~\ref{fig:vor}.

\subsubsection{A Merged Configuration}
Let $\nu$ be the probability distribution on $\mathcal{C}^{0,2}$ given by sampling points from a Poisson point process with density one in the right half of the plane, and placing points in a fixed triangular lattice in the left half of the plane. Let $\hat{\nu}={V^2}_*\nu.$ $\hat{\nu}$ samples embedded graphs that look like Voronoi graph in the right half of the plane, and a hexagonal lattice in the left half. A finite region of such a graph is shown in~\ref{fig_halfHex2}.

Graphs sampled from $\hat{\nu}$ have a smooth geometric cloth with probability one, but that cloth does not equal $\hat{\nu}.$ Let $\mu_1$ be the probability distribution of Voronoi graphs in the planar with intensity 1, and let $\mu_2$ be the cloth of the hexagonal lattice in the left half of the plane. Then, with probability one the cloth of graph $G$ sampled from $\hat{\nu}$ is given by 
\s{\mathcal{P}\paren{G}\paren{B}=\frac{\mu_1+\mu_2}{2}}
 
\subsubsection{Penrose tilings}
As is the case for Benjamini-Schramm graph limits~\cite{2012lovasz}, Penrose tilings provide an example of embedded graphs with interesting geometric cloths. A full analysis of their properties is beyond the scope of this paper but can be derived based on the work of de Bruijin~\cite{1981bruijin,1981bruijin2}.

\subsection{Local Geometric Convergence}
\label{sec_GeoConv}
We are now ready to define local geometric convergence.
\begin{Definition}
\label{defn_geoConv}
A sequence $\set{G_i}$ of embedded graphs in $\mathcal{G}^n$ is \textbf{locally geometrically convergent} if each $G_i$ has smooth geometric cloth, and $s^*\mathcal{P}_r\paren{G_i}$ converges weakly to a limit $\Xi_r$ for all $r\in \N.$ 
\end{Definition}

For example, if $G$ is an infinite graph with smooth geometric cloth, the sequence $G\cap B_i$ is locally geometrically convergent. It follows immediately from the definition that if $\set{G_i}\rightarrow \Xi$ in the local geometric sense and $f$ is a local geometric property then
\s{\lim_{i\rightarrow\infty}\widebar{f_{G_i}}=\int_{\mathcal{G}^n}f\;d\Xi}
Furthermore, if the $\R^n$ translation action is ergodic with respect to a geometric graph limit $\Xi,$ then $\Xi$-almost all $G$ have local geometric properties whose averages equal the limiting values of the sequence averages. We will consider the implications of this hypothesis in more detail in Section~\ref{sec_hom}

The smooth topology of topological types is not particularly nice nor natural. However, weak convergence of a uniformly separating sequence on $\mathcal{G}^n$ implies local geometric convergence:
\begin{theorem}
\label{thm_convergence}
Let $\set{G_i}$ be a uniformly separating sequence of graphs in $\mathcal{G}^n.$ The following are equivalent:
\begin{enumerate}
\item $\set{G_i}$ is locally geometrically convergent
\item $\set{s^*\mathcal{P}_r\paren{G_i}}$ is weakly convergent for all $r\in\N.$
\item $\mathcal{P}_r\paren{G_i}$ is weakly convergent for all $r\in \N.$
\item $\mathcal{P}\paren{G_i}$ is weakly convergent
\item The average values of all local geometric properties $f,$ $\widebar{f_{G_i}}$ converge.
\end{enumerate}
\end{theorem}
\begin{proof}
The second condition is the definition of local geometric convergence. The following three are equivalent because of Propositions~\ref{localProp} and~\ref{unifProp}. The last condition is equivalent to the first because a uniformly separating sequence is tight.
\end{proof}

The following question can be viewed as an analogue to the Aldous-Lyons conjecture~\cite{2012lovasz} in the local topological context: 

\begin{question}
Is any translation-invariant probability distribution on $\mathcal{G}^n$ the geometric cloth of an embedded graph $G$? More specifically, does this hold if the distribution $\mu$ is the the local geometric limit of a sequence of embedded graphs $\set{G_i}$? 
\end{question}

\subsection{Scale-free Convergence}
\label{sec_scaleFree}
The concept of local geometric convergence is too strong for curvature flow on graphs - computational simulations indicate that scale-dependent properties of change as graphs evolve.  To define a notion of scale-free convergence, we need a definition of a characteristic length scale:

\begin{Definition}
Let $\rho_{x,\alpha}:\mathcal{G}^n\rightarrow \mathcal{G}^n$ be the map induced the the dilation of $\R^n$ centered at $x$ by a factor $\alpha.$ A \textbf{characteristic length scale} of graphs is a continuous function $g:Y\rightarrow\R^+$ defined on a dilation-invariant subset $Y$ of $\mathcal{G}^n$ that scales with dilations: 
\s{g\circ\rho_{x,\alpha}=\alpha g \;\;\;\forall \alpha\in\R^+,x\in\R^n}

An embedded graph $G$ has a \textbf{well-defined characteristic length scale} $\widebar{g}_G$ if $g$ is a characteristic length scale and
\s{0<\widebar{g}=\lim_{r\rightarrow\infty} \int_{\mathcal{G}^{n}}g\;P\paren{G,B_r}=\int_{\mathcal{G}^{n,k}}g\;\mathcal{P}\paren{G}<\infty}
\end{Definition}
A characteristic length scale is necessarily unbounded, so we will need to impose some hypotheses on an infinite embedded graph to ensure it has a well-defined characteristic lengthscale. For example, the the function that assigns to a graph the infimal $r$ so that the probability that a ball of radius $r$ fully contains an edge of $G$ is a well-defined characteristic lengthscale for graphs with smooth geometric cloth whose maximum edge length is bounded above.

\begin{Definition}
Suppose $\set{G_i}$ is a sequence of embedded graphs in $\mathcal{G}^n$ which share a characteristic length-scale $g.$ Then $\set{G_i}$ converges in the \textbf{scale-free local geometric sense} if the sequence of rescaled graphs 
\s{\set{\rho_{\frac{1}{\widebar{g}_{G_i}} G_i}}}
is locally geometrically convergent, where $\rho_\alpha$ is the dilation of $\R^n$ by $\alpha$ centered at the origin. A \textbf{scale-free local geometric property} of cell complexes is a local geometric property that is dilation-invariant.
\end{Definition}

\subsection{Relation to Local Topological Convergence}
\label{sec_geoAndBS}

Local topological properties are not in general local geometric properties, for two reasons: small topological neighborhoods may be arbitrarily large geometric ones, and local topological properties are normalized by topological quantities rather than by volume. However, one may place hypotheses on the tail of the edge length distribution and the variation of vertex density under which local geometric convergence implies local topological convergence.  

There are several ways to define a topological cloth of an infinite embedded graph, given certain consistency assumptions. These notions coincide for finite graphs. We consider one such definition here, but very similar statements can be made about others. Let $G\in\mathcal{G}^n$ be an embedded graph, and let $I_r\paren{G}$ be the graph consisting of all edges and vertices of $G$ fully contained in $B_r.$ Recall that $G$ is said to have an inner regular topological cloth if $\set{I_r\paren{G}}$ is locally topologically convergent.

Let $e\paren{G,s,M}$ be the percentage of edges in $I_s\paren{G}$ of length greater than $M.$ Also, define the vertex density function $\rho\paren{G,s,r,M}$ to be the percentage of vertices $v\in I_s\paren{G}$ so that $B_r\paren{v}$ contains more than $M$ vertices of $G.$  

\begin{proposition}
Let $\set{G_j}$ be a locally geometrically convergent sequence of graphs in $\mathcal{G}^n,$ whose vertex degrees are bounded above by $D$ (with probability one with respect to local distributions $\mathcal{P}_r\paren{G_j}$). Also, let $V_j\paren{B_r}$ be the vertex set of $I_r\paren{G_j}.$ 

Assume there  is a constant $C_1>0$ and bounded, real-valued functions $M\paren{\epsilon}$ and $N\paren{\epsilon}$ so that for all $\epsilon>0$ and all sufficiently large $s$ and $j$
\begin{enumerate}
\item $\abs{\frac{\abs{V_j\paren{B_s}}}{\vol{B_s}}-C_1}<\epsilon$
\item $e\paren{G_j,s,M\paren{\epsilon}}<\epsilon$
\item $\rho\paren{G_j,s,r,N\paren{\epsilon}}<\epsilon$
\end{enumerate}then each $G_j$ has an inner-regular topological cloth and
\s{\lim_{j\rightarrow\infty} \lim_{r\rightarrow\infty}I\paren{G_j,B_r}}
 converges in the Benjamini-Schramm sense.
\end{proposition}The proof is long and technical, but straightforward, and we omit it here. 

\section{Homogeneous Graphs}
\label{sec_hom}
A homogeneous graph is an embedded graph with smooth geometric cloth for which the translation action on $\R^n$ is ergodic. This will be the key hypothesis for the universality conjectures we propose for curvature flow on graphs in Section~\ref{sec_conj}. Here, we study the properties of homogeneous graphs. In particular, we show that the local geometric properties of a homogeneous graph have well-defined averages over a large class of sequences of increasing subsets of $\R^n.$ Note that a homogenous graph is necessarily infinite. 

\subsection{Preliminaries: Ergodic Theory}

\begin{Definition}
Let $\paren{X,\textit{F},\nu}$ be a measure space, and suppose a group $H$ acts on $X.$ The action of $H$ is \textbf{measure-preserving} if $\nu\paren{hB}=\nu\paren{B}$ for all $B\in\textit{F}$ and all $h\in H.$ In this case $\nu$ is said to be \textbf{$H$-invariant}. The action is \textbf{ergodic} if every subset of $X$ such that $hX=X\;\forall h\in H$ has measure 0 or 1.
\end{Definition}We require a generalized version of the Pointwise Ergodic Theorem due to Lindenstrauss~\cite{2001lindenstrauss}.
\begin{theorem}
(Lindenstrauss)
Let $T$ be a locally compact, amenable group acting on a measure space $\paren{X,\textit{F},\nu}$ with left-invariant Haar measure $\mu,$ and suppose $\set{F_n}$ is a sequence of compact subsets of $T$ that is
\begin{enumerate}
\item F{\o}lner: for all $t\in T$
\s{\lim_{i\rightarrow \infty} \frac{\mu\paren{tF_i\triangle F_i}}{\mu\paren{F_i}}=0}
where $\triangle$ denotes the symmetric difference.
\item Tempered: There exists a $C>0$ such that for all $n,$ 
\s{\mu\paren{\cup_{k\leq n} F_k^{-1}F_{n+1}}\leq C\mu\paren{F_{n+1}}}
\item \s{\abs{F_n}\geq n}
\end{enumerate}
If the action of $H$ is measure-preserving, then for any $f\in L^1\paren{\nu}$ there is a well-defined, $H$-invariant average $\widebar{f}\in L^1\paren{\nu}$ so that for $\nu$-almost all $x$
\s{\lim_{i\rightarrow\infty}\frac{1}{\mu\paren{F_i}}\int_{F_i}f\paren{gx}\;d\mu\paren{g}=\widebar{f}\paren{x}}
In addition, if the action is ergodic
\s{\widebar{f}\paren{x}=\int_{X}f\;d\nu\paren{x}}
\end{theorem}

\subsection{Homogeneous Graphs}
\begin{definition}
\label{defn_homo}
An embedded graph $G\in\mathcal{G}^n$ is \textbf{homogeneous} if it has a smooth geometric cloth, and the translation action of $\R^n$ is ergodic with respect to $\mathcal{P}\paren{G}.$
\end{definition} The following is an immediate consequence of the definition of ergodicity:
\begin{proposition}
If $G$ is homogeneous and $f$ is a local geometric property then $\mathcal{P}\paren{G}$-almost all $H$ have $\widebar{f_G}=\widebar{f_H}.$
\end{proposition} That is, the geometric cloth of a homogeneous graph $G$ is a probability distribution sampling graphs with the same local geometric properties as $G.$ Note the resemblance to the topological cloth.  

One might hope that if $f$ is a local geometric property and $G$ is a homogeneous graph, then $f$ would have a well-defined average over any tempered F{\o}lner sequence. A counterexample is presented in Section~\ref{sec_counterexample}. However, the local properties of homogeneous graphs have well-defined averages over a smaller set of averaging sequences:

\begin{Definition}
An increasing sequence $\set{F_i}_{i\in\N}$ of compact, convex sets in $\R^n$ is an \textbf{admissible averaging sequence} if 
\begin{enumerate}
\item$\lim_{i\rightarrow\infty}\vol{F_i}=\infty$
\item $\limsup_{i\rightarrow\infty}\frac{\vol{B\paren{F_i}}}{\vol{F_i}}<\infty$
where where $B\paren{F_i}$ is the smallest ball centered at the origin containing $F_i.$ 
\end{enumerate}
\end{Definition}

\begin{proposition}
Let $X$ be a completely regular space with $\R^n$ action. For $x\in X$ and any bounded subset $K$ of $\R^n,$ define the empirical measure $P\paren{x,K}$ to the pushforward of the normalized Lebesgue measure on $K$ via the map $\phi_x:\R^n\rightarrow X$ given by $\phi_x\paren{a}=ax.$ If $\lim_{r\rightarrow\infty}P\paren{X,B_r}$ converges weakly to a probability distribution $\sigma$ for which the $\R^n$ action is ergodic, then $\lim_{i\rightarrow\infty}P\paren{X,F_i}$ converges weakly to the same limit for any admissible averaging sequence $\set{F_i}.$ 
\end{proposition}
\begin{proof}

Let $\set{F_i}$ be an admissible averaging sequence, and $\mu$ be the standard Lebesgue measure on $\R^n.$ Also, let $f:X\rightarrow\R$ be a bounded continuous function, and let $\bar{f}$ be the expected value of $f$ for $x\in X,$ taken over a sequence of balls centered at the origin. We will show that $f$ has the same expected value for $f$ for any admissible averaging sequence. To do so, we will replace $f$ with an averaging function $f_r$:
\s{f_r\paren{x}=\frac{1}{\vol{B_r}}\int_{B_r}f\paren{ax}\;d\mu\paren{a}.}
Also, let
\s{\hat{F}_i^r=\set{a\in F_i : B_r\paren{a}\subset F_i}}
Then
\begin{align*}
\int_{F_i}f_r\paren{ax}\;d\mu\paren{a}&\\
&=\frac{1}{\vol{B_r}}\int_{F_i}\int_{B_r}f\paren{abx}\;d\mu\paren{b}\;d\mu\paren{a}\\
&=\frac{1}{\vol{B_r}}\int_{B_r}\int_{\hat{F}_i^r}f\paren{abx}\;d\mu\paren{a}d\mu\paren{b}+\int_{F_i-\hat{F}_i^r}f_r\paren{ax}\;d\mu\paren{a}\\
&=\int_{\hat{F}_i^r}f\paren{ax}\;d\mu\paren{a}+O\paren{\text{vol}_{n-1}\paren{\partial F_i}}\\
&=\int_{F_i}f\paren{ax}\;d\mu\paren{a}+O\paren{\text{vol}_{n-1}\paren{\partial F_i}}
\end{align*}
because $f$ is bounded and the volume of $F_i-\hat{F}_i^r$ is $O\paren{\text{vol}_{n-1}\paren{\partial F_i}}.$ Therefore if the error of approximating $f$ by $f_r$ is denoted
\s{\text{err}\paren{F_i,r}=\frac{1}{\vol{F_i}}\abs{\int_{F_i}f_r\paren{ay}\;d\mu\paren{a}-\int_{F_i}f\paren{y}},} then $text{err}\paren{F_i,r}$ goes to zero as $i\rightarrow\infty.$ 

For any $\epsilon>0,$ define a sequence of open subsets of $X:$
\s{A_{s,\epsilon}=\set{x\in X:\abs{\frac{1}{\vol{B_s}}\int_{B_s}f\paren{ax}\;d\mu\paren{a}-\widebar{f}}<\epsilon}}
Ergodicity implies that for any $\delta>0$ and sufficiently large $s,$ 
\s{\sigma\paren{A_{s,\epsilon}\paren{f}}>1-\delta} Let $\epsilon>0$ and choose $s$ such that $\sigma\paren{A_{s,\epsilon}}>1-\epsilon/2.$ Because $\mu_i$ converges weakly to $\mu$ as $r\rightarrow\infty$, the Portmanteau lemma implies that for all sufficiently large $r,$ 
\s{P\paren{x,B_r}\paren{A_{s,\epsilon}}>1-\epsilon}
The second hypothesis in the definition of an admissible averaging sequence implies that
\s{\limsup_{i\rightarrow\infty}{\frac{\text{vol}_{n-1}\paren{\partial F_i}}{\text{vol}_n\paren{F_i}}}<\infty}
via the Blaschke selection theorem~\cite{2007gruber}. Thus, for sufficiently large $i,$
\s{P\paren{x,F_i}\paren{A_{s,\epsilon}}>1-\epsilon}It follows that there is an $I_1$ so that for all $i>I_1$ we can write $F_i=E_i\cup D_i$ where $E_{i}\subset  A_{s,\epsilon}$ and $\sup\abs{f}\mu\paren{D_i}<\epsilon\mu\paren{F_i}$ This also implies $\frac{\mu\paren{E_i}}{\mu\paren{F_i}}>1-\frac{\epsilon}{\sup\abs{f}}$. 

Let $M=\sup\abs{f}$, and choose $I_2$ so that $\text{err}\paren{F_i,s}<\epsilon$ for all $i>I_2.$ Then for all $i>\max\paren{I_1,I_2},$ 

\begin{align*}
\abs{\frac{1}{\vol{F_i}}\int_{F_i}f\paren{ax}\;d\mu\paren{a}-\widebar{f}}\\
&=\abs{\frac{1}{\vol{F_i}}\int_{F_i}f_s\paren{ay}\;d\mu\paren{a}+\text{err}\paren{F_i,s}-\widebar{f}}\\
&\leq\abs{\frac{1}{\vol{F_i}}\big(\int_{E_i}f_s\paren{ay}\;d\mu\paren{a}+\int_{D_i}f_s\paren{ay}\;d\mu\paren{a}\big)-\widebar{f}}+\epsilon\\
&\leq \abs{\frac{\vol{E_i}}{\vol{F_i}}\frac{1}{\vol{E_i}}\int_{E_i}f_s\paren{ay}\;d\mu\paren{a}-\widebar{f}}+M\frac{\mu\paren{D_i}}{\mu\paren{F_i}}+\epsilon\\
&\leq\abs{\frac{1}{\vol{E_i}}\int_{E_i}f_s\paren{ay}\;d\mu\paren{a}-\widebar{f}}+\frac{\epsilon}{M\vol{E_i}}\abs{\int_{E_i}f\paren{ay}\;d\mu\paren{a}}+2\epsilon\\
&\leq4\epsilon
\end{align*} So $P\paren{x,F_i}\rightarrow\sigma,$ weakly, as desired.
\end{proof} The following is an immediate consequence:
\begin{theorem}
If $G$ is a homogeneous graph then then $s^*\mathcal{P}\paren{G,F_i}$ converges weakly to $\mathcal{P}\paren{G}$ for any admissible averaging sequenced. Furthermore, $\mathcal{P}\paren{G}$-almost all graphs have local geometric properties whose averages are equal to those of $G,$ and are well-defined over any tempered F{\o}lner sequence.
\end{theorem}

\subsection{A Counterexample}
\label{sec_counterexample}
\begin{figure}[ht]
\center
\includegraphics[width=.4\textwidth]{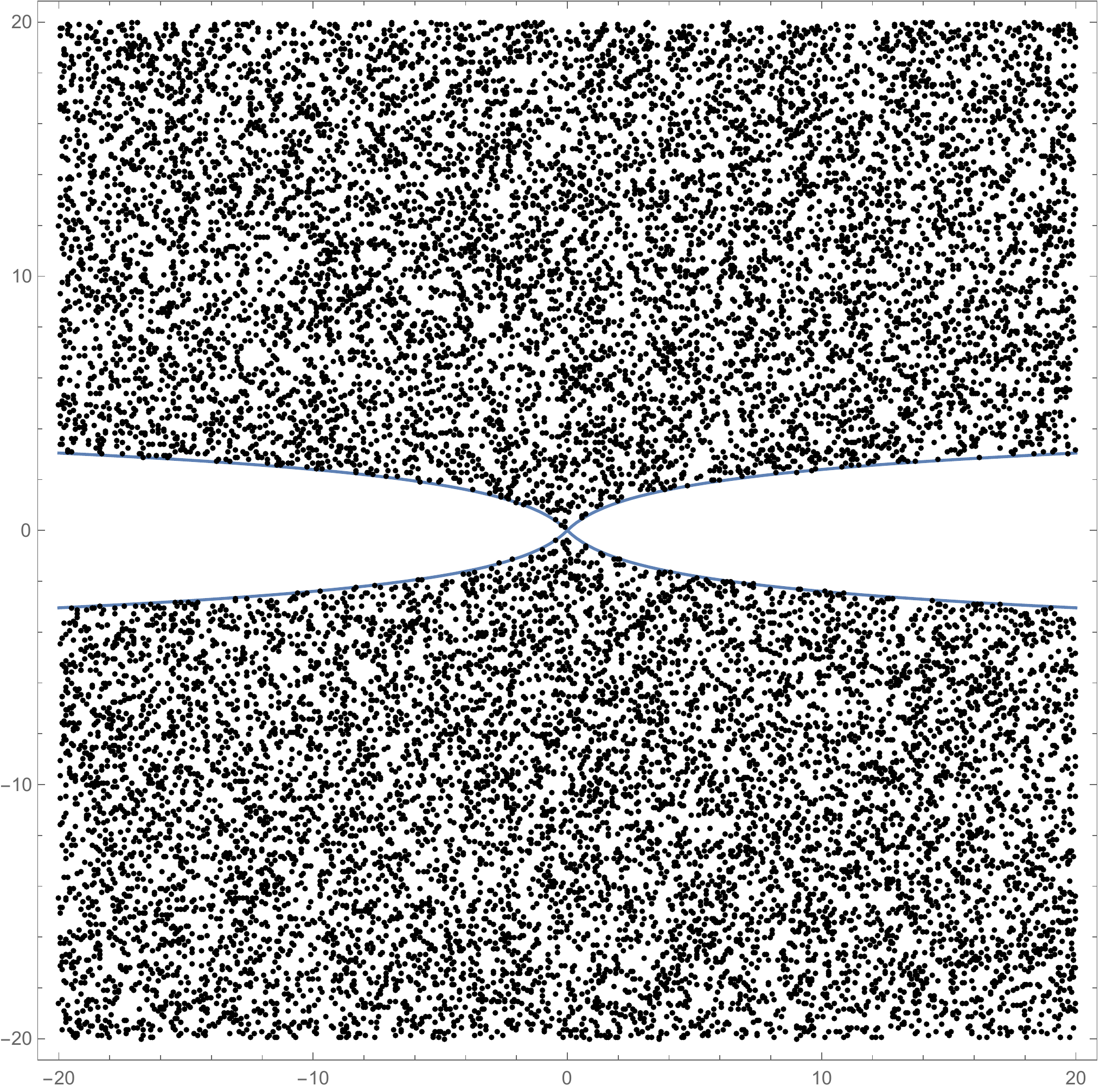}
 \caption{\label{fig_counterex} A Poisson distributed point collection, with points from the region $M$ removed.}
\end{figure} 
We will provide an example of a homogeneous locally finite point collection whose local properties do not have well-defined averages over all tempered F{\o}lner sequences. The Voronoi diagram of such a point collection will also share that property. Note that the concepts of geometric cloth and homogeneity extend automatically to locally finite point collections.

Let $\mu$ be the Poisson Point Process with intensity one on $\R^2.$ Let $M\subset\R^2$ be the region bounded by the inequality
\s{\text{log}\paren{\abs{x}+1}\leq\abs{y}}
Also for any $w\in Y,$ let $\hat{w}$ be the locally finite point collection obtained from $w$ by deleting all points contained in $M.$ An example is shown in Figure~\ref{fig_counterex}.  The volume fraction
\s{\frac{\vol{M\cap B_r}}{\vol{B_r}}\leq\frac{4r\text{log}\paren{r+1}}{\pi r^2}}
goes to zero as $r\rightarrow\infty.$ It follows from the pointwise ergodic theorem that for $\mu$-almost all $w,$ $\hat{w}$ has a well-defined geometric cloth.

Let $E_r$ be the rectangle centered at the origin bounded by the lines $x=\pm r$ and $y=\pm\text{log}\paren{r+1}.$ $E_r$ is a nested sequence of convex sets whose inradius goes to $\infty,$ so it is a tempered F{\o}lner sequence. However, the volume ratio
\s{\frac{\vol{M\cap E_r}}{\vol{E_r}}=\frac{4\int_{0}^{r}\text{log}\paren{s+1}\;d\mu\paren{s}}{4r\text{log}\paren{r}}=\frac{4\paren{\paren{r+1}\text{log}\paren{r+1}-r}}{4r\text{log}\paren{r+1}}}
goes to one as $r\rightarrow\infty,$ so $P\paren{\hat{w},E_r}$ cannot converge weakly to $\mu.$

\section{The Universality Conjectures}
\label{sec_conj}
We use the concepts of Local Geometric Convergence and Local Topological Convergence to make several onjectures about the long-term behavior of graphs evolving by the network flow of curvature flow on embedded graphs in $\R^n.$

\begin{conj}[Universality Conjecture for Local Topological Convergence]
\label{steady_state_hypothesisb1}
There exists a probability distribution $\sigma_n$ on the space of countable, connected graphs with a root vertex specified such that any network flow $G\paren{t}$ with homogeneous initial condition $G\paren{0}\in\mathcal{G}^n$ converges in the local topological sense to $\sigma_{\Omega}$ or to a stationary state as $t\rightarrow\infty.$
\end{conj} Note that $G\paren{t}$ might converge topologically to a stationary state in a local sense without ever being stationary itself. For example, the flow starting with an infinite regular hexagonal lattice with a single edge removed is never stationary, but always has the the local statistics of a stationary state. Assuming the previous conjecture, let $\sigma_{n}$ be the universal probability distribution on graphs in $\R^n.$ We conjecture all that all embedded graphs with homogeneous topological cloth, as defined below, converge to  $\sigma_{n}$ for some $n.$ This allows, for example, an embedded graph in $\R^3$ to converge to the universal distribution for planar embedded graphs.

\begin{Definition}
\label{defn_tophomo}
$G\in\mathcal{G}^n$ has a \textbf{homogeneous topological cloth} if there is a probability distribution $\mathcal{S}\paren{G}$ on $\mathcal{G}^\bullet$ such that the sequences $I\paren{C,G_i}$ and $O\paren{C,G_i}$ converge in the Benjamini-Schramm sense to $\mathcal{S}\paren{G}$ for every admissible averaging sequence $F_i.$
\end{Definition}

\begin{conj}[Classification of Orbits]
\label{steady_state_hypothesis10}
Suppose $G\in\mathcal{G}^n$ is an embedded graph with a homogeneous topological cloth and a well-defined time evolution under curvature flow $G\paren{t}.$ Then either $G\paren{t}$ converges topologically to $\sigma_{m}$ for some $m\leq n,$ or it converges to a stationary state.
\end{conj} We would not be surprised if $\sigma_{n}=\sigma_{m}$ for all $m,n\geq 3.$

\begin{conj}[Universality Conjecture for Local Geometric Convergence]
\label{steady_state_hypothesisa1}
There exists a probability distribution $\Xi_n$ on $\mathcal{G}^n$ such that  that any network flow $G\paren{t}$ with homogeneous initial condition $G\paren{0}\in\mathcal{G}^n$converges in the scale-free local geometric sense  to $\Xi_{n}$ or a stationary state as $t\rightarrow\infty.$
\end{conj} It is conjectured that curvature flow on graphs never results in an edge with multiplicity greater than one~\cite{2014ilmanen}. Under this assumption, the corresponding universality conjecture in terms of local geometric convergence of graphs with multiplicity in $\check{\mathcal{G}}^n$ would be equivalent to the above.

\subsection{Statement in Terms of Invariant Measures}

We defined local geometric convergence in terms of individual graphs in order to find an analogue of Benjamini-Schramm graph limits for embedded graphs, and to develop a language that would be amenable to computation. However, a perhaps more traditional way to state a universality conjecture about the behavior of graphs evolving under curvature flow is in terms of invariant measures:

\begin{conj}
\label{steady_state_hypothesis_ergodic}
Let $\hat{\Omega}$ be curvature-flow on embedded graphs in $\R^n,$ rescaled to keep a characteristic length scale equal to one. There exists a unique $\hat{\Omega}$-invariant, homogeneous probability measure $\nu$ on $\mathcal{G}^n$ such that $\nu$-almost all graphs do not evolve to a stationary state under $\hat{\Omega}.$ Furthermore, $\hat{\Omega}$ is ergodic with respect to $\nu.$

Such a measure $\nu$ can be obtained by choosing any homogeneous initial condition $G\in\mathcal{G}^n$ that does not converge to a stationary state, and taking the limit of $\mathcal{P}\paren{G\paren{t}}$ as $t\rightarrow\infty.$ 
\end{conj}

\section{Computational Evidence}
\label{sec_data}
\label{sec_comp}
\label{sec_computation}
We test the universality conjectures stated in the previous section for embedded graphs in two dimensions We simulated curvature flow of graphs embedded in $\mathbb{T}^2$ using Jeremy Mason's implementation of the algorithm developed by Lazar et al in~\cite{2010lazar}. It is based on the von Neumann - Mullins Relation, which gives the rate of change of the area of a component of the complement in terms of the number of edges on its boundary. 

We use two kinds of random initial conditions for our simulations. The first is the Voronoi diagram of Poisson distributed points, as depicted in Figure~\ref{fig:vor}. The other, shown in \ref{perturbed_lattice} is a perturbed hexagonal lattice. It is generated by perturbing the vertices of the dual triangular lattice, then computing a Voronoi diagram.

\subsection{Testing for Local Topological Convergence}
\label{sec_swatchComp}
\label{sec_computations}
\begin{figure}
\center
\includegraphics[height=7cm]{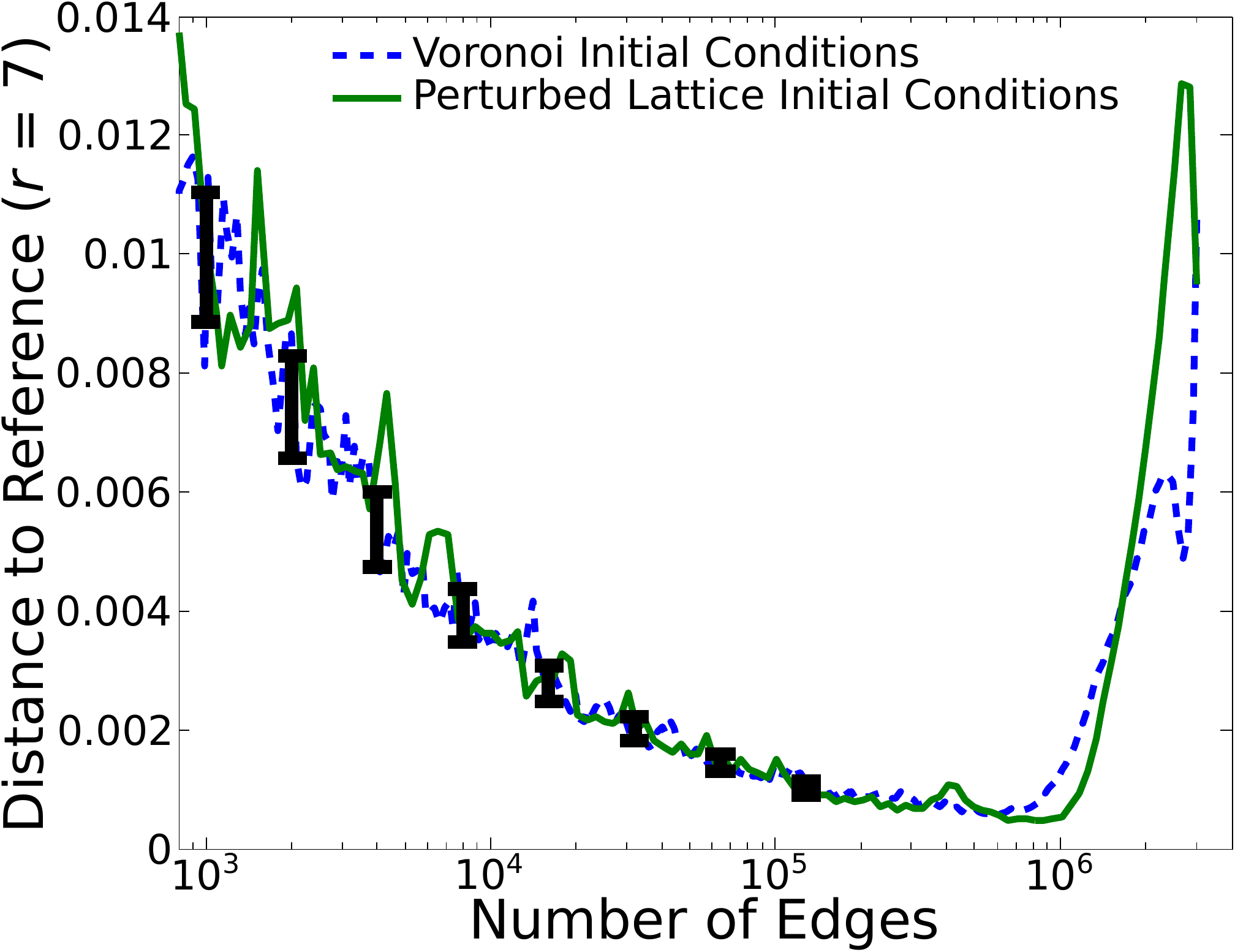}
\caption{\label{fig_2D1}Distances of a simulations starting with both initial conditions to the candidate universal state. Error bars show the standard deviation of the distance of a representative subsample of the candidate state to the candidate state.}
\end{figure}

We test the for local topological convergence by tracking the distance $d_r$ of a simulated graph throughout its evolution to a candidate universal state. The candidate universal state is taken from a large simulation for which all measured properties have converged. We sample swatches of a given radius at each vertex, and use the program \textit{Nauty} to classify them up to graph isomorphisms~\cite{nauty}.

Separate implementations of these algorithms were programmed by both the author and Jeremy Mason. The two programs work for general regular cell complexes, not just those of dimension one. More details on the implementation is included in~\cite{2016schweinhart}.

Figure~\ref{fig_2D1} shows the distance from from two simulations - one starting from a Voronoi graph and the other a perturbed hexagonal lattice - to a large candidate universal state. The candidate universal state was taken from a different simulation starting with a Voronoi graph with $6.0 \times 10^7$ edges. The evolution is parametrized by the number of edges in the graph, a decreasing function of time. Thus, these figures are read from right to left. All simulations start at initial conditions far from the candidate state, before decreasing toward the universal state. Finite size effects cause the distances to increase toward the end of the simulation.

The distance from both simulations to the candidate universal state decreases very rapidly. However, in none of these systems does the distance to the candidate universal state go to zero. To quantify the finite size effects, we constructed representative subsamples of the candidate universal state. The error bars in Figure~\ref{fig_2D1} show one standard deviation above and below the mean of the distance of a representative subsample of a certain size to the candidate universal state. The two simulations are within the bars from between 100,000 and 1000 edges, indicating that the deviation of the distance from zero is explainable due to finite size effects.

\cite{2016schweinhart} contains a more detailed analysis of these results, as well as similar ones for the network flow in three dimensions.

\subsection{Testing for Local Geometric Convergence}
\begin{figure}
\center
\subfloat[]{%
	\label{HEMD_1}{%
		\includegraphics[width=.45\textwidth]{%
			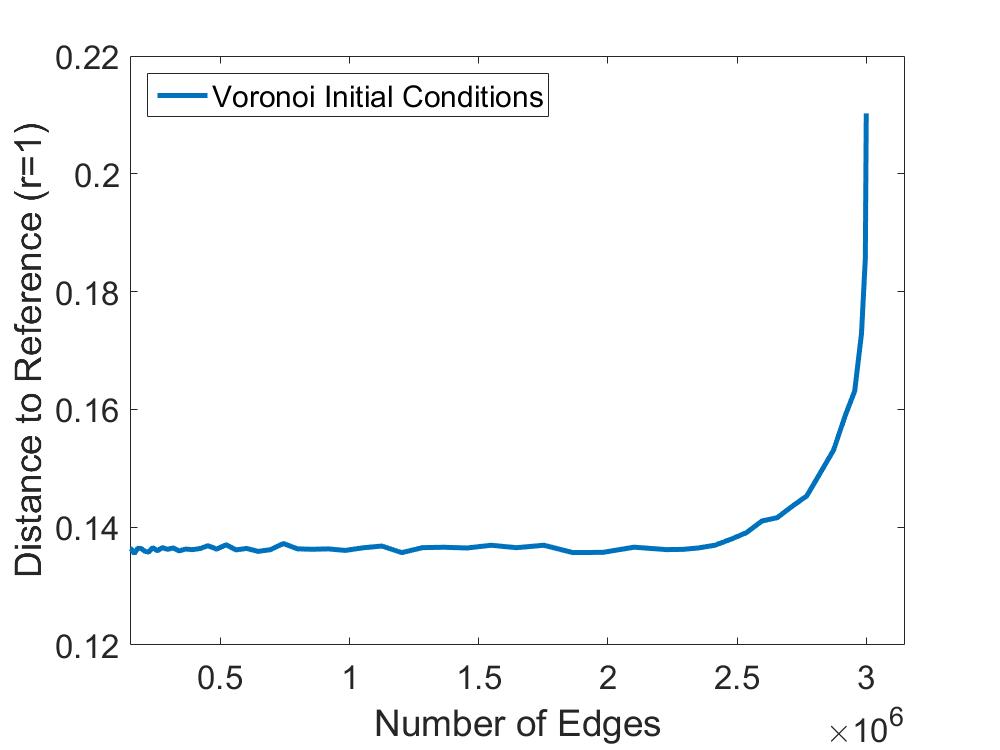}}}
\hspace{20pt}
\subfloat[]{%
	\label{HEMD_2}{%
		\includegraphics[width=.45\textwidth]{%
			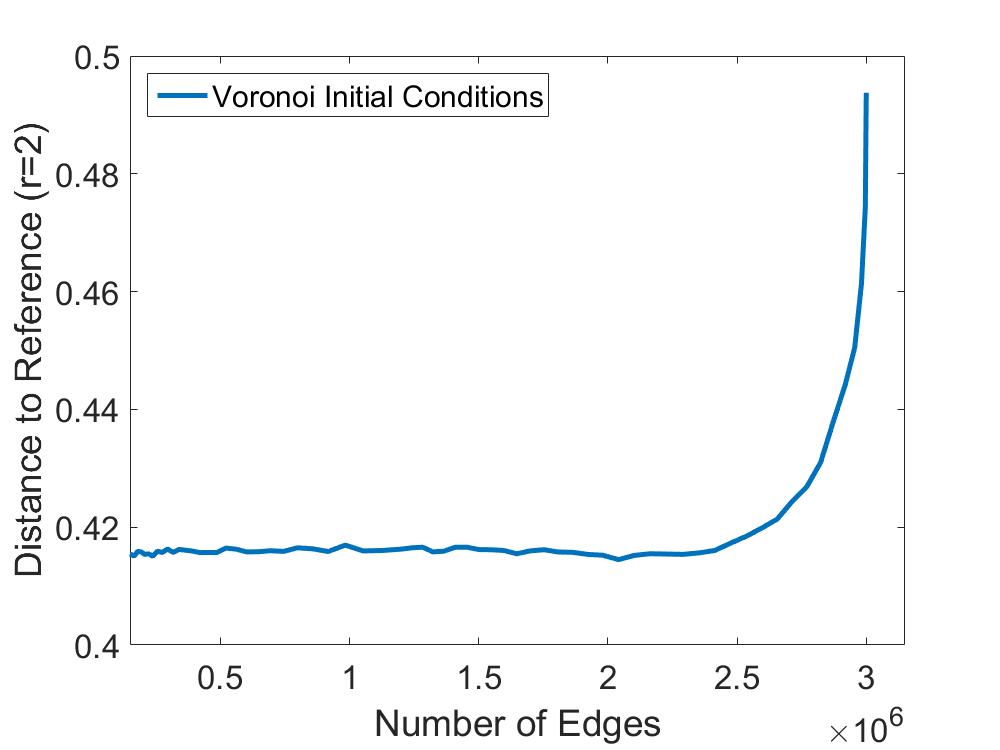}}}
\caption{\label{Fig_HEMD}Plots showing the local Hausdorff EMD from a simulation to a candidate universal state with (a) $r=1$ and (b) $r=2.$ Data from perturbed lattice initial conditions will also be included in a final draft.}
\end{figure}

To test for local geometric convergence, we rescale a sequence of graphs $\set{G_i}$ so that the average edge length of each equals one, and test the weak convergence and uniform separability of $\mathcal{P}_r\paren{G_i}$ for several values of $r.$ 

To check weak convergence, we estimate the earth mover's distance (first Wasserstein distance) induced by the local Hausdorff distance. We sample $10^4$ random $r$-balls in each $G_i$ and a candidate universal state $G,$ compute a distance matrix in $d_L$ between these samples, and use the transport package in $R$~\cite{transport} to determine the earth mover's distance. Plots of this distance for $r=1$ and $r=2$ are shown in Figure~\ref{Fig_HEMD}. As before, the number of edges decreases with time, so the plots are meant to be read from left to right. In each, the distance to the candidate state decreases rapidly. Note that this convergence is faster than in Figure~\ref{fig_2D1}, indicating that the local properties at small radius converge much faster than those at larger radii.   

For $G\in\mathcal{G}_r^2,$ let $\Delta\paren{G}$ be the supremal $\delta$ so that $G$ is $\delta$-thickenable. Note that the curvature of $G$ is bounded by $\Delta\paren{G}^{-1},$ by the definition of $\delta$-thickenability. Figure~\ref{fig_thicken} depicts a graph $H\in\mathcal{G}_2^2$ sampled from the candidate universal state for which $\Delta\paren{H}=.266.$ 

To check uniform separation, we computed the probability distribution $\Delta$ for $10^5$ random $r-$balls in $G_i.$  Note that the conditions in the definition of $\delta$-thickenability have direct analogues for discretized graphs. For example, we use the discrete global radius of curvature~\cite{1999gonzales} as a stand-in for the maximum $\epsilon$ so that an edge has an $\epsilon$-tube. To study the tail of $\mu_i,$ we plotted percentile curves of these distributions as a function of the number of edges. This is shown in Figure~\ref{Fig_thickenability_check} for $r=1.$ Each plot appear to approach to limiting values in the mid-range of the evolution, and are bounded uniformly away from zero. They are certainly uniformly bounded away from $0.$ This supports the hypothesis the sequence is uniformly separated. Interestingly, the percentage of samples of the candidate universal state where each of the following was the limiting factor for $\Delta$ is approximately 89.4\% for the minimum distance of an edge or vertex to the boundary, 5.2\% for the minimal separation between vertices, 2.2\% for the global radius of curvature, and .6\% for the distance between edges outside of a neighborhood of the vertices. The remaining 2.6\% of the samples were empty. This is unsurprising because edges meet at vertices at angles very close to $2\pi/3,$ and the grains in the candidate universal state appear to be quite ``round'' (as can be seen in Figure~\ref{ss_2D}).

\begin{figure}
\center
\includegraphics[height=7cm]{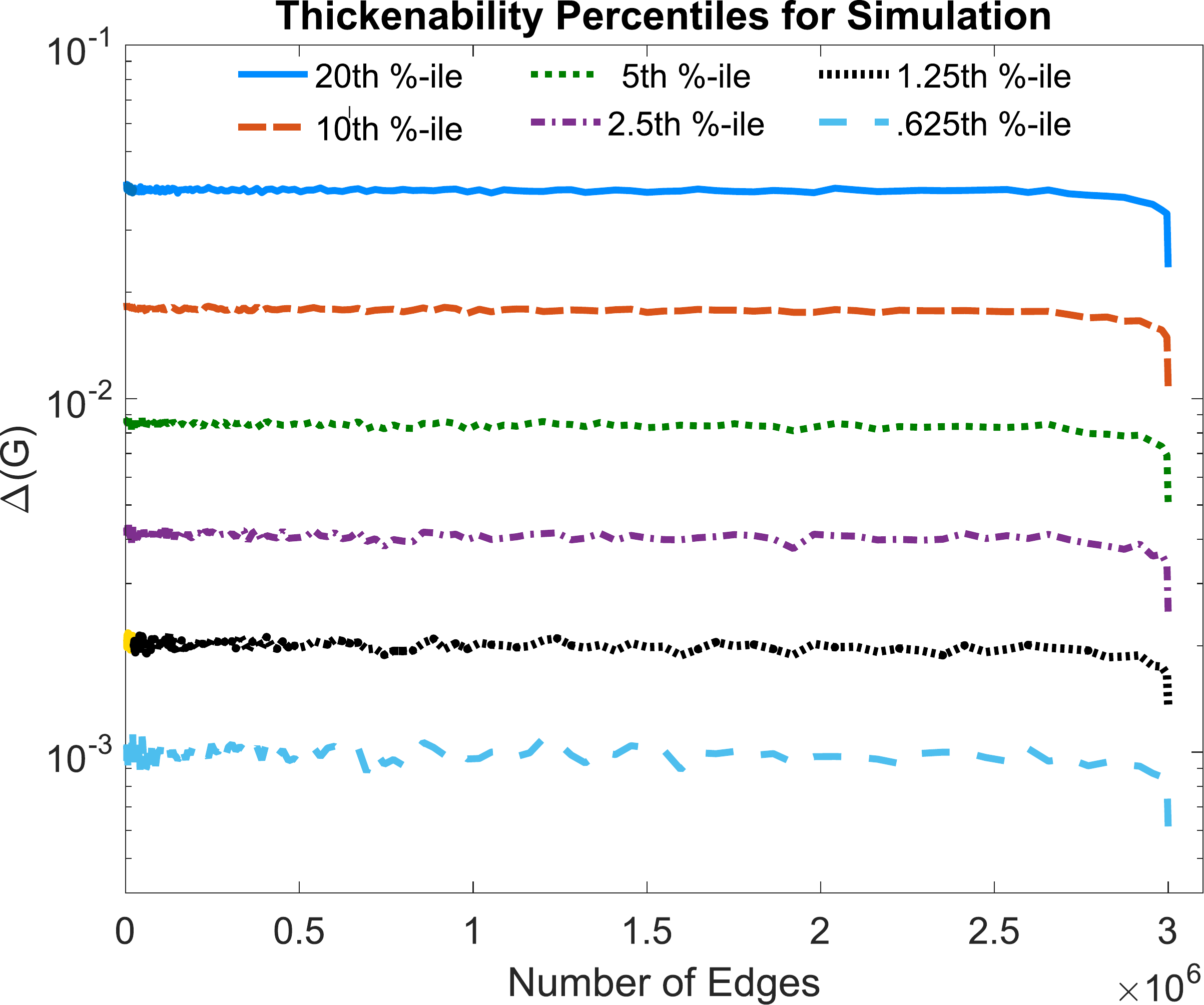}
\caption{\label{Fig_thickenability_check}Percentile values of $\Delta\paren{G}$ for a graph evolving by curvature flow.}
\end{figure}

We used the the spatial searching module of CGAL to improve the speed of both the computation of the local Hausdorff distance and of $\Delta\paren{G}$~\cite{CGAL,CGALspatial}. In the future, we hope to further optimize computation of the Hausdorff earth mover's distance by taking $d_L$ mod rotations. 

\section{Conclusion}
Curvature flow on graphs and cell complexes is a process of great mathematical interest, with important physical applications in materials science. We defined notions of local topological convergence and local geometric convergence for embedded graphs in $\R^n,$ and studied their properties. Using these concepts, we stated several universality conjectures for curvature flow on graphs,  formalizing observations from experiments and computer simulations of these systems.

It is our hope that this paper will stimulate interest in these problems, eventually leading to a mathematically rigorous understanding of this beautiful phenomenon.

\section*{Acknowledgments}
The author would like to thank Robert MacPherson and Jeremy Mason and in particular for discussions that led to the ideas in this paper. The method of swatches was originally developed in collaboration with them in~\cite{2016schweinhart}. Professor MacPherson was the author's PhD advisor, and suggested curvature flow on graphs as a system of interest.  The author would also like to thank Tom Ilmanen, Ryan Peckner, and Philippe Sosoe for interesting and useful discussions. 

This research was supported by an NSF Mathematical Sciences Postdoctoral Research Fellowship and, before that, the Center of Mathematical Sciences and Applications at Harvard University.

\appendix
\section{Proofs from Section~\ref{sec_graphTop}}
\label{appendix}
\propOne*
\begin{proof}
Let $G_i=r\paren{\check{G}_i}$ and $G=r\paren{G}.$ Also, let $\hat{G}_i=r\paren{G_i}\cup \partial B_r$ and $\hat{G}=r\paren{G}\cup \partial B_r.$ $\set{\hat{G}_i}$ is a sequence of compact subsets of $\R^n,$ so there is subsequence $\set{\hat{H}_i}$ which converges in the Hausdorff distance to a limit $X_0.$ Let $X=X_0\cap B_r.$ We will show that $X=G.$

Let $x\in X$ and suppose $x\notin G.$ Then there is an $\epsilon>0$ so that $B_\epsilon\paren{x}\subset B_r$ and $B_\epsilon\paren{x}\subset B_r\cap G= \emptyset.$ Choose a continuous function $g:\R^n\times \mathbb{RP}^{n-1} \rightarrow \brac{0,1}$ that is one on $B_{\epsilon/2}\paren{x}\times \mathbb{RP}^{n-1}$ and zero outside of $B_\epsilon\paren{x}\times \mathbb{RP}^{n-1}.$ There is a positive integer $I$ so that for all $i>I,$ there is a point of $G_i$ within distance $\epsilon/4$ of $x.$  
\s{\mathcal{H}^1\paren{H_i\cap B_{\epsilon/2}\paren{x}}>\frac{\epsilon}{2}}
because $H_i$ is connected. Therefore for sufficiently large $i$
\s{\int_{\R^n\times \mathbb{RP}^{n-1}}g\mu_{\check{G}_i}>\frac{\epsilon}{2}}
which is a contradiction, because $\mu_{\check{G}_i}\rightarrow \mu_{\check{G}}$ and the supports of $g$ and $\mu_{\check{g}}$ are disjoint. Thus $x\in G.$

Let $x\in G$ and $\epsilon>0$ be sufficiently small so that $B_\epsilon\paren{x}$ is contained in the open ball of radius $r.$ Let $g:\R^n\rightarrow \brac{0,1}\times \mathbb{RP}^{n-1}$ be a positive function that is one on $B_{\epsilon/2}\paren{x}\times \mathbb{RP}^{n-1}$ and zero outside $B_\epsilon\paren{x}\times \mathbb{RP}^{n-1}.$ Because $G$ is an embedded graph, the mass of $G\cap B_{\epsilon/2}\paren{x}$ is at least $\epsilon,$ so
\s{\int_{\R^n\times \mathbb{RP}^{n-1}}g\mu_{\check{G}}>\epsilon} For sufficiently high $i,$
\s{\abs{\int_{\R^n\times \mathbb{RP}^{n-1}}g\mu_{\check{G}}-\int_{\R^n\times \mathbb{RP}^{n-1}}g\mu_{\check{G}_i}}<\frac{\epsilon}{2}}
so
\s{\mu_{\check{G}_i}\paren{B_{\epsilon}\paren{x}}\geq \int_{\R^n\times \mathbb{RP}^{n-1}}g\mu_{\check{G}_i}\geq \int_{\R^n\times \mathbb{RP}^{n-1}}g\mu_{\check{G}}-\frac{\epsilon}{2}\geq \frac{\epsilon}{2}}

and $G_i\cap B_{\epsilon}\paren{x}\neq \emptyset.$ $\epsilon$ can be made arbitrarily small, so there a sequence of points $\set{x_i}$ with $x_i\in G_i$ and $x_i\rightarrow x.$ Therefore $x\in X$ and $X=G,$ as desired.
\end{proof} 

\begin{lemma}
\label{shellLemma}
If an edge has a tubular neighborhood of radius $\delta,$ then its intersection with any ball of radius less than $\delta/2$ is connected. Furthermore, the intersection of such an edge a sphere of radius less than $\delta/2$ has at most two components.
\end{lemma}

\begin{proof}
Let $e$ be an edge parametrized by $\psi:\brac{0,1}\rightarrow\R^n$, and let $\pi:T_{\delta}\paren{e}\rightarrow e$ be the projection map of the tubular neighborhood to $e,$ viewed as a subset of the normal bundle of $e.$ Note that if $y\in e,$ then $\pi^{-1}\paren{y}$ is the set of all points in $T_\delta\paren{e}$ closer to $y$ than any other point of $e.$

Let $\epsilon<\delta<2,$ and $x\in \R^n.$ If $B_\epsilon\paren{x}\cap e \neq \emptyset$
\s{B_{\epsilon}\paren{x} \subset B_{\delta}\paren{\pi\paren{x}}\subset T_{\delta}\paren{e}}
Let $a_1=\psi^{-1}\paren{\pi\paren{x}}$ and suppose that $a_1<a_2<a_3$ with $\psi\paren{a_3}\in B_{\delta}\paren{\pi\paren{x}}.$ The line segment between $\psi\paren{a_3}$ and $x$ is contained $B_{\delta}\paren{\pi\paren{x}},$ and is disconnected by removing the fiber $\pi^{-1}\paren{\psi\paren{a_2}}.$ It follows that there is a $y\in\pi^{-1}\paren{\psi\paren{a_2}}$ with
\s{d\paren{x,\psi\paren{a_3}}=d\paren{y,\psi\paren{a_3}}+=d\paren{y,x}}
so
\s{d\paren{\psi\paren{a_2},x}\leq d\paren{\psi\paren{a_2},y}+d\paren{y,x} < d\paren{y,\psi\paren{a_3}}+d\paren{y,x} =d\paren{x,\psi\paren{a_3}}}
Therefore, the distance to $x$ is monotonically increasing on $\psi^{-1}\paren{B_{\delta}\paren{\pi\paren{x}}}\cap\brac{a_1,1},$ and a similar argument shows it is monotonically decreasing on $\psi^{-1}\paren{B_{\delta}\paren{\pi\paren{x}}}\cap\brac{0,a_1}.$ The desired statements follow immediately.
\end{proof}

\lemmaTwo*
\begin{proof}
Let $G\in\mathcal{G}^n$ and $B_r$ be any ball in $\R^n.$ $B_r$ contains finitely many vertices and edges, so we may choose $\delta_0$ so that any $\delta\leq\delta_0$ satisfies all but perhaps the third condition.

We will examine how pairs of edges approach vertices. Let $e_1$ and $e_2$ be edges of $G$ meeting at a vertex $v.$ Translate $G$ if necessary so that $v=0\in\R^n.$ Let $c_1\paren{t}$  and $c_2\paren{t}$ be unit-speed parametrizations of $e_1$ and $e_2$ with $c_1\paren{0}=c_2\paren{0}=0=v.$ Then if $T_1$ and $T_2$ are the unit tangent vectors (which are distinct by the definition of an embedded graph), $N_1$ and $N_2$ are the unit normal vectors, and $\kappa_1$ and $\kappa_2$ are the curvatures of $c_1$ and $c_2$ at $v$ (with orientations given by the parametrization), we have that
\s{c_1\paren{t}=tT_1+\frac{t^2\kappa_1}{2}N_1+o\paren{t^2}=tT_1+O\paren{t^2}}
and
\s{c_2\paren{t}=tT_2+\frac{t^2\kappa_2}{2}N_2+o\paren{t^2}=tT_2+O\paren{t^2}}
where the orders are taken with respect to the limit $t\rightarrow 0.$ Let $e_1\paren{r}$ and $e_2\paren{r}$ be the values of $c_1\paren{t}$ and $c_2\paren{t}$ at their first intersection with the sphere of radius $r.$ Then
\begin{align*}
\norm{e_1\paren{r}-e_2\paren{r}}&\\
&\geq\norm{c_1\paren{r}-c_2\paren{r}}+\norm{c_1\paren{r}-e_1\paren{r}}+\norm{c_2\paren{r}-e_2\paren{r}}\\
&=r\norm{T_1-T_2}+O\paren{r^2}
\end{align*}
and
\s{\lim_{r\rightarrow 0}\frac{\norm{e_1\paren{r}-e_2\paren{r}}}{r^2}\geq\lim_{r\rightarrow 0}\frac{r\norm{T_1-T_2}+O\paren{r^2}}{r^2}=\infty>1}
Therefore there exists an $\xi>0$ so that for all $\epsilon<\xi,$ the distance between $c_1$ and $c_2$ when they first intersect ball of radius $\epsilon$ is greater than $\epsilon^2.$

Lemma~\ref{shellLemma} implies that $G\cap \partial B_{\epsilon}\paren{v}$ is a finite collection of points, one for each edge adjacent to $v.$ $G$ has finitely many vertices, each of finite degree, so by the previous argument we may chose $\delta_1$ so that any $\delta\leq \delta_1$ satisfies the third condition.

Therefore if $\delta<\min\paren{\delta_0,\delta_1},$ $G$ is $\delta$-thickenable.
\end{proof} 

\begin{lemma}
Let $G^1, G^2\in\mathcal{G}^n_r$ be $\delta_0$-thickenable and $\delta<\frac{\delta_0}{5}.$ If $d_L\paren{G^1,G^2}<\frac{\delta^4}{18},$ there is a canonical combinatorial isomorphism of $G^1$ with $G^2$ so that
\begin{enumerate}
\item Each vertex of $v^1$ $G^1$ is paired with the unique vertex $v^2$ of $G^2$ with $d\paren{v^1,v^2}<\delta.$
\item If $V^1$ is the union of the union of the vertex set of $G^1$ with $\partial B_r,$ and $V^2$ is the corresponding set for $V_2,$ $e^1$ of $G^1$ is matched with the unique edge $e^2$ of $G^2$ so that
\s{d_H\paren{e^1\cap \paren{V^1_{2\delta}\cup V^2_{2\delta}}^c,e^2\cap \paren{V^1_{2\delta}\cup V^2_{2\delta}}^c}<\frac{\delta^2}{2}}
\end{enumerate}
\end{lemma}
\begin{proof}

First, we will pair the vertices of $G^1$ with those of $G^2.$ Let $v^1$ be a vertex of $G^1,$ $\set{e_i}$ be edges of $G^1$ adjacent to $v^1,$ and
\s{A=B_{\frac{\delta^2}{2}}\paren{v^1}-B_{\frac{\delta^2}{3}}\paren{v^1}}
$\delta$-thickenability implies that the only edges of $G^1$ that intersect $A$ are precisely $\set{e_i},$ and that 
\s{T_{\frac{\delta^4}{18}}\paren{e_i}\cap T_{\frac{\delta^4}{18}}\paren{e_j} \cap A \neq 0 \Rightarrow e_i=e_j}
$d_H\paren{G^1,G^2}<\frac{\delta^4}{18},$ it follows that $G^2\cap A$ has at least $\abs{\set{e_i}}\geq 3$ components. This together with Lemma~\ref{shellLemma} implies that more than one edge of $G^2$ intersects $A.$ The edges of a $\delta$-thickenable graph can approach each other within to within distance $\frac{\delta^2}{2}$ only inside the $\delta$-ball of a vertex (or much closer to the boundary), so there must be a vertex of $v_2$ $G^2$ in $\widebar{B}_{\delta}\paren{v}.$ Furthermore, at most one vertex of a $2\delta$-thickenable graph can be contained within a ball of radius $\delta,$ so this vertex is unique. Repeating the argument with $G^1$ and $G^2$ swapped gives the desired bijection of vertices.

$G^2_{\delta^2/2}\cap \paren{V^2_{\delta}}^c$ is the disjoint union of the $\frac{\delta^2}{2}$-tubes about the edges of $G^2.$ These tubes are connected because of the last condition in the definition of $\delta$-thickenability. The previous paragraph implies that
\s{V^1_\delta\subset V^2_{2\delta}\subset V^1_{4\delta}}
So if $e^1$ is an edge of $G^1,$
\s{e^1\cap \paren{V^1_{2\delta}}^c\subset G^2_{d_{L}\paren{G^1,G^2}}\cap\paren{V^2_{\delta}}^c}
$e^1\cap \paren{V^1_{2\delta}}$ is connected (because of condition 3 in the definition of $\delta$-thickenability) and $d_{L}\paren{G^1,G^2}<\frac{\delta^2}{2}$ so there is a unique edge $e^2$ of $G^2$ so that
\s{e^1\cap \paren{V^1_{2\delta}}\subset N_\frac{\delta^2}{2}\paren{e^2}\cap\paren{V^2_{\delta}}^c}
Repeating this argument with $G^1$ and $G^2$ swapped shows that the edges of of the two graphs are paired bijectively so that
\s{d_H\paren{e^1\cap \paren{V^1_{2\delta}\cup V^2_{2\delta}}^c,e^2\cap \paren{V^1_{2\delta}\cup V^2_{2\delta}}^c}<\frac{\delta^2}{2}}

If $e^1$ is an edge of $G^1$ that is adjacent to a vertex $v^1$ of $G^1,$ $v^2$ is the vertex of $G^2$ paired with it, and $x=e^1\cap \partial B_{4\delta}\paren{v^1},$ then
\s{d\paren{e^2,v^2}\leq d\paren{e^2,v^1}+d\paren{v^2,v^1} < d\paren{e^2,x}+5\delta<\delta_0}
so $e^1$ is adjacent to $v^2.$ Therefore, the bijection between the edges of $G^1$ and $G^2$ is a graph isomorphism, as desired.

\end{proof}

\begin{lemma}
The set $Y_\delta\subset \mathcal{G}^n_r$ of $\delta$-thickenable graphs is closed and equals the disconnected union of finitely many sets of constant topological type. A sequence of graphs in $Y_\delta$ converges if and only if vertices converge to vertices and edges converge to edges in the local Hausdorff distance.
\end{lemma}

\begin{proof}
The previous lemma implies that the subset of $Y_\delta$ having a given topological type is both closed and open in that set, so $Y_\delta$ is the disconnected union of these sets. Only finitely many topological types intersect $Y_\delta$ because the number of vertices and edges of a graph in that set is bounded.

Let $\set{G_i}$ be a sequence of $\delta$-thickenable embedded graphs in $\mathcal{G}^n_r$ converging to a limit $G.$ $G$ is $\delta_0$-thickenable for some $0<\delta_0<1,$ so by the previous lemma we may find an $i$ so that for $i>I,$ there is a canonical graph isomorphism $\phi_i:G_i\rightarrow G.$ Furthermore, if $c$ is an edge (or vertex) of $G,$ then $\psi_i^{-1}\paren{c}\rightarrow c$ in the local Hausdorff distance.

It remains to be shown that $G\in Y_\delta.$ All conditions in the definition of $\delta$-thickenability except for the existence of a tubular neighborhood of radius $2\delta$ second follow quickly from from Lemma~\ref{shellLemma}, and the convergence of edges to edges and vertices to vertices.

The global radius of curvature of a curve is the minimal $\epsilon$ so that that curve does not have a tubular neighborhood of radius $\epsilon~\cite{1999gonzales}.$ It is given by
\s{\rho\paren{e}:=\min\set{r\paren{x,y,z}: x,y,z\in e, x\neq y, x\neq z, y\neq z}}
where $r\paren{x,y,z}$ is the circumradius of the triangle with vertices $x,y,$ and $z.$

Let $e$ be an edge of $G,$ and $\set{e_i}$ be a sequence of edges converging to $e.$ We will show that $e$ has a global radius of curvature less than or equal to $2\delta.$ By taking closures in $\widebar{B_r}$, we may assume that the edges $e$ and $e_i$ are closed. Let $f:\brac{0,1}\rightarrow\R$ and $f_i:\brac{0,1}\rightarrow\R$ be constant-speed parametrizations of $e$ and $\set{e_i},$ respectively, with $f_i\paren{0}\rightarrow f\paren{0}$ and $f_i\paren{1}\rightarrow f\paren{1}.$ Each $e_i$ has a tubular neighborhood of radius $2\delta,$ so its length is bounded above by $L=\frac{V_n\paren{r}}{V_{n-1}\paren{2\delta}},$ where $V_i\paren{x}$ is the volume of the $i$-dimensional ball of radius $x.$ Therefore the derivatives of the $f_i'$ are uniformly bounded, and the Arzela-Ascoli theorem implies that there is a convergent subsequence of $\set{f_i},$ which must convergence to $f.$ This is true for all convergent subsequences, so $f_i\rightarrow f.$ 

Let $f\paren{a_1}=y_1, f\paren{a_2}=y_2$ and $f\paren{a_3}=y_3$ be three distinct points of $e.$ Let 
\s{\xi=\min\paren{d\paren{x_1,x_2},d\paren{y_1,y_3},d\paren{y_2,y_3}}} and choose $J$ so that $i>J$ implies that $d\paren{f_i\paren{a_j},f\paren{a_j}}<\xi/4$ for $j\in\set{1,2,3}.$ The function $r:\R^n\times\R^n\times\R^n\rightarrow\R\cup\infty$ sending a triplet of points to the circumradius of the triangle formed by those points is continuous on the set $\set{x_1,x_2,x_3\in\R^n\times\R^n\times\R^n: x_1\neq x_2, x_1\neq x_3, x_2\neq x_3}$~\cite{WeissteinCircumradius}:
It follows that $r\paren{f_i\paren{a_1},f_i\paren{a_2},f_i\paren{a_3}}$ is defined for $i>J$ and that
\s{\lim_{i\rightarrow\infty}r\paren{f_i\paren{a_1},f_i\paren{a_2},f_i\paren{a_3}}\rightarrow r\paren{f\paren{a_1},f\paren{a_2},f\paren{a_3}}}
so $r\paren{f\paren{a_1},f\paren{a_2},f\paren{a_3}}\geq 2\delta.$ As this is true for all triplets of distinct points on $e,$ it follows that the global radius of curvature of $G$ is greater than or equal to $2\delta$ and $G$ is $\delta$-thickenable, as desired.
\end{proof}

\begin{lemma}
Let $Z_{S,\delta,D}\subset\ X_{D,\delta}\subset \mathcal{G}_r^{n}$ be the set of $\delta$-thickenable graphs of a fixed topological type $S,$ the $n-1$ Frenet-Serret curvatures of whose edges are bounded in magnitude by $D.$ The smooth and Hausdorff metric topologies coincide on $Z_{S,\delta,D},$ and it is compact. 
\end{lemma}

\begin{proof}
We will show that  $Z_{S,\delta,D}$ is compact in the smooth topology. This will imply that the smooth and Hausdorff metric topologies coincide on this set, because a bijective map from a compact space into a Hausdorff space is necessarily a homeomorphism. Let $\set{G_i}$ be a sequence of  graphs in $Z_{S,\delta,D},$ and $\set{e_i}$ be any sequence of of edges of the graphs $\set{G_i}.$ If any of the $\set{e_i}$ are not closed in $B_r,$ replace each edge with its closure in $\widebar{B}_r.$ Parametrize the $e_i$ by constant-speed maps $f_i:\brac{0,1}\rightarrow \widebar{B}_r\paren{0}.$ The Frenet-Serret theorem implies that all higher derivatives of $\set{f_i}$ can be expressed in terms of their generalized curvatures, and so are uniformly bounded at each order. It follows from the Arzela-Ascoli Theorem that we can find a subsequence of $\set{f_i},$ $\set{f_{i_j}},$ that converges smoothly to function $f\in C^{\infty}\brac{\brac{0,1},\widebar{B}_r}.$ Furthermore,
\s{f_i'=L\paren{e_i}>\frac{\delta}{2}}
so $f'\neq 0,$ and $f$ is an embedding. Therefore the edges $\set{e_{i_j}}$ of the subsequence $\set{G_{i_j}}$ converge smoothly to a limit edge $E_1.$

If $S$ has more than one edge, we can find a sequence $\set{\hat{e}_{i_j}}$ of edges of $\set{G_{i_j}}$ so that
\s{d_L\paren{\hat{e}_{i_j},e_{i_j}}>\delta/2}
 Using the same argument as in the previous paragraph, we can refine the sequence further so that $\hat{e}_{i_k}$ converge smoothly to a limit edge $E_2$ distinct from $E_1.$ Continuing this for each integer less than or equal to the number of edges of $S$ yields a subsequence of $\set{G_i}$ converging in smooth topology to a limit $\hat{G},$ which must be in  $Z_{S,\delta,D}.$ Thus $Z_{S,\delta,D}$ is sequentially compact, and therefore compact.
\end{proof}

\propTwo*
\begin{proof}
The first two statements follow immediately from the preceding lemmas, and the third does from the corresponding lemmas for the varifold topology. We do not include them here for the purpose of brevity, but note that they are a straightforward consequence of~\ref{prop_HausdorffVarifold}, the Allard compactness theorem, and the fact that the multiplicity of a convergent sequence of edges in a $\delta$-thickenable graph equals the multiplicity of their limit.
\end{proof}

\bibliographystyle{plain}
\bibliography{Bibliography}

\end{document}